\documentclass[amssymb,amsfonts,refcheck,12pt,verbatim,righttag]{amsart}

\usepackage[all]{xy}
\usepackage{bbm}

\usepackage{amsmath}
\usepackage{amsfonts}
\usepackage{amssymb}
\usepackage{mathrsfs}
\usepackage{savesym}

\usepackage{graphicx}

\usepackage{graphics,color}

\usepackage[colorlinks,citecolor=blue,pagebackref,urlcolor=black,hypertexnames=false]{hyperref}

\numberwithin{equation}{section} 
\theoremstyle{plain}

\newtheorem{thm}{Theorem}[section]
\newtheorem{lem}[thm]{Lemma}
\newtheorem{pro}[thm]{Proposition}

\newtheorem{de}[thm]{Definition}
\newtheorem{rem}[thm]{Remark}

\def\R {{\Bbb R}}
\def\N {{\Bbb N}}

\def\ba{{\bf a}}
\def\bb{{\bf b}}

\def\bi{{\bf i}}
\def\bj{{\bf j}}
\def\ss{\pmb{\mathfrak{s}}}
\def\tt{\pmb{\mathfrak{t}}}

\DeclareMathOperator*{\esssup}{ess\,sup}
\DeclareMathOperator*{\essinf}{ess\,inf}

\usepackage{float}
\begin{document}
\baselineskip 14pt
\title{Dimension estimates for $C^1$ iterated function systems and  repellers. Part II}

\author{De-Jun Feng}
\address[De-Jun Feng]{Department of Mathematics\\ The Chinese University of Hong Kong\\ Shatin,  Hong Kong\\ }
\email{djfeng@math.cuhk.edu.hk}

\author{K\'aroly Simon}
\address[K\'aroly Simon]{Budapest University of Technology and Economics, Department of Stochastics, Institute of Mathematics and
MTA-BME Stochastics Research Group,
 1521 Budapest, P.O.Box 91, Hungary} \email{simonk@math.bme.hu}


\date{}

\begin{abstract}
This is the second part of our study
of the dimension theory of $C^1$ iterated function systems (IFSs) and repellers on $\mathbb{R}^d$.
In the first part \cite{FengSimon2020} we proved that the upper box-counting dimension of the attractor of every $C^1$ IFS on $\R^d$ is bounded above by its singularity dimension, and
the upper packing dimension of every ergodic invariant measure associated with this IFS is bounded above by its Lyapunov dimension. Here we introduce a generalized transversality condition (GTC) for  parameterized families of $C^1$ IFSs, and show that  if the GTC is satisfied then the dimensions of the IFS attractor and of the ergodic invariant measures are given by these upper bounds, for almost every (in an appropriate sense) parameter.  Moreover we verify  the GTC for some parametrized families of $C^1$ IFSs on $\R^d$.
\end{abstract}

\maketitle

\section{Introduction}
\label{S-1}
The present paper is a continuation of our work in \cite{FengSimon2020} for studying
the dimension theory of $C^1$ iterated function systems (IFSs) and repellers.

One of the fundamental problems in fractal geometry and dynamical systems is to compute various fractal dimensions of attractors of IFSs and associated invariant measures. The corresponding problem has been well understood when the underlying IFSs consist of similitudes or conformal maps satisfying certain separation conditions (see e.g.~\cite{Hutchinson1981, Hochman2014, Bowen1979,  Ruelle1982, GatzourasPeres1997, Patzschke1997, przytycki2010conformal}). The problem becomes substantially more difficult when the underlying IFSs are non-conformal. In the last 3 decades, many significant progresses have been achieved for affine IFSs, see e.g.~\cite{Bedford1984, McMullen1984, Falconer88, Kaenmaki04, JordanPollicottSimon07, FengShmerkin2014, DasSimmons2017, FalconerKempton2018, Barany2019,HochmanRapaport2019} and the references in the survey papers \cite{Chen_2010, Falconer2013} and a coming book \cite{BSS2021}.

 In contrast to the extensive studies on affine IFSs, there have been relatively few results on those IFSs which are neither conformal nor affine.  In 1994, Falconer \cite{Falconer94} introduced a quantity (known as the singularity dimension) in terms of sub-additive topological pressure, and showed that it is an upper bound for the upper box-counting dimension of repellers of $C^2$ expanding maps satisfying a ``bunching'' condition. Later in 1997, Zhang \cite{Zhang97} proved that this upper bound holds for the Hausdorff dimension of repellers of arbitrary $C^1$ expanding maps. We remark that the results of Falconer and Zhang extend directly to the IFS setting.  Recently, Cao, Pesin and Zhao \cite{CaoPesinZhao19} also gave an upper bound for the upper box-counting dimension of repellers of $C^{1+\alpha}$ expanding maps satisfying  a certain dominated
splitting property. However that upper bound depends on the   splitting involved and
is usually strictly larger than the singularity dimension.  In \cite{FengSimon2020} the authors proved that the singularity dimension is an upper bound of the upper box-counting dimension of the attractor of every $C^1$ IFS or the repeller of every $C^1$ expanding map, improving the aforementioned results in \cite{Falconer94, Zhang97, CaoPesinZhao19}. The authors also established a measure analogue of this result, that is, the upper packing dimension of  every ergodic invariant measure associated with a $C^1$ IFS or repeller is bounded above by its Lyapunov dimension, improving an earlier result of Jordan and Pollicott \cite{JordanPollicott08} for the upper Hausdorff dimension of measures. The reader is referred to Section~\ref{S-2.3} for the definitions of singularity dimension and Lyapunov dimension.

In \cite{Hu1998} Hu computed the box-counting dimension of repellers of  $C^2$ maps on $\R^2$ which  have an invariant strong unstable foliation along which they expand more strongly than in the complementary directions. Very recently, Falconer, Fraser and Lee \cite{FalconerFraserLee2020} computed the $L^q$-spectra of Bernoulli measures associated with a class of planar IFSs consisting of $C^{1+\alpha}$ maps for which the Jacobian is a lower triangular matrix subject to a domination condition and satisfying the rectangular open set condition.  As a corollary they obtained a formula for the box-counting dimension of the attractors of such plannar IFSs. In another recent paper \cite{JurgaLee2020}, Jurga and Lee proved that, under slightly stronger assumptions, these Bernoulli measures (and more generally, quasi-Bernoulli measures) on the attractors are exact dimensional with dimension given by a Ledrappier-Young type formula.  In earlier related works, Bedford and Urba\'{n}ski \cite{BedfordUrbanski1990} calculated the box-counting and Hausdorff dimensions of the attractors of a very special class of planar nonlinear triangular $C^{1+\alpha}$ IFSs (of which the attractors are curves),   Manning and Simon \cite{ManningSimon07} and B\'{a}r\'{a}ny \cite{Barany2009} studied the sub-additive pressure associated to nonlinear $C^{1+\alpha}$ IFSs  whose maps have triangular Jacobians.

In this paper, we introduce a generalized transversality condition (GTC) for  parametrized families of $C^1$ IFSs on $\R^d$, and show that  if the GTC is satisfied then for almost every (in an appropriate sense) parameter,  the Hausdorff and box-counting dimensions of the IFS attractor are indeed given by the singularity dimension, and the dimension of ergodic invariant measures on the attractor is given by its Lyapunov dimension.  Moreover, we will verify  the GTC for several classes of translational families of $C^1$ IFSs.

Before formulating our results precisely,  we first recall some basic notation and definitions. By a $C^1$ IFS on a compact set $Z\subset\R^d$ we mean a finite collection $\mathcal F=\{f_i\}_{i=1}^\ell$ of self-maps on $Z$,  such that there exists an open set $U\supset Z$ so that each $f_i$ extends to a  $C^1$-diffeomorphism $f_i: U\to f_i(U)\subset U$ with
$$
\rho_i:=\sup_{x\in U}\|D_xf_i\|<1,
$$
where $D_xf$ stands for the differential of $f$ at $x$ and $\|\cdot\|$ is the standard matrix norm (i.e., $\|A\|$ is the largest singular value of $A$).

Let $K$ be the attractor of the IFS $\mathcal F$, that is,  $K$ is the unique non-empty compact subset of $Z$ such that
\begin{equation}
\label{e-IFS}
K=\bigcup_{i=1}^\ell f_i(K)
\end{equation}
(cf.~\cite{Hutchinson1981}).

Let $(\Sigma,\sigma)$ be the one-sided full shift over the alphabet $\{1,\ldots, \ell\}$.  Let $\Pi: \Sigma\to K$ denote the corresponding coding map associated with the IFS $\mathcal F$, that is,
\begin{equation}
\label{e-code}
\Pi(\bi)=\lim_{n\to \infty}f_{i_1}\circ \cdots\circ f_{i_n}(0), \qquad \bi=(i_n)_{n=1}^\infty.
\end{equation}
It is well known that $\Pi$ is continuous and surjective (\cite{Hutchinson1981}).  For a $\sigma$-invariant Borel probability measure $\mu$ on $\Sigma$, let $\Pi_*\mu$ denote the push-forward of $\mu$ by $\Pi$, that is, $\Pi_*\mu(E)=\mu(\Pi^{-1}(E))$ for each Borel subset $E$ of $\R^d$.

For a Borel probability measure $\xi$ on $\R^d$, we call
$$
\underline{d}_\xi(x)=\liminf_{r\to 0}\frac{\log \xi(B(x,r))}{\log r}\quad \mbox{ and }\quad \overline{d}_\xi(x)=\limsup_{r\to 0}\frac{\log \xi(B(x,r))}{\log r}
$$
the {\it lower and upper local dimensions of $\xi$ at $x$}, where $B(x,r)$ stands for the closed ball centered at $x$ of radius $r$.
Moreover, we call
$$
\underline{\dim}_H\xi=\essinf_{x\in {\rm spt}(\xi)}\underline{d}_\xi(x) \quad \mbox{ and }\quad \overline{\dim}_P\xi=\esssup\limits_{x\in {\rm spt}(\xi)}\overline{d}_\xi(x)
$$
the {\it lower Hausdorff dimension and upper packing dimension of $\xi$}, respectively. If   $\underline{\dim}_H\xi=\overline{\dim}_P\xi$, we say that $\xi$ is {\it exact dimensional} and write $\dim \xi$ or  $\dim_H \xi$ for this common value.

To introduce the notion of GTC,  let  $\ell\geq 2$ and let ${\mathcal F}^t=\{f_1^t,\ldots, f_\ell^t\}$, $t\in {\Omega}$, be a parametrized family of $C^1$ IFSs defined on a common compact subset $Z$ of $\R^d$, where $({\Omega}, \rho)$ is a separable metric space, such that  the following two conditions hold: \begin{itemize}
\item[(C1)] The maps $f_i^t$ have a common Lipschitz constant $\theta\in (0,1)$, that is
\begin{equation}
\label{e-f1}
|f^t_i(x)-f^t_i(y)| \leq \theta |x-y|
\end{equation}
for all $1\leq i\leq\ell$, $t\in {\Omega}$ and $x,y\in Z$.
\item[(C2)]The mapping $t\mapsto f^t_i(x)$ is continuous over ${\Omega}$ for every given $x\in Z$ and $1\leq i\leq \ell$.
\end{itemize}

For each $t\in {\Omega}$, let $K^t$ denote the attractor of $\mathcal F^t$, and let $\Pi^t:\; \Sigma\to \R^d$  denote the coding map associated with the IFS ${\mathcal F}^t$.
Due to the conditions (C1) and (C2),  the mapping $(t,\bi)\mapsto \Pi^t(\bi)$ is continuous over the product space ${\Omega}\times \Sigma$.

For $t\in {\Omega}$,  $r>0$ and ${\bf i}\in \Sigma_*:=\bigcup_{n=0}^\infty \{1,\ldots, \ell\}^n$, set
\begin{equation}
\label{e-f1.0}
Z_{\bf i}^t(r)=
\inf_{x\in \Sigma} \min\left \{ \frac{r^k}{ \phi^k (D_{\Pi^t x}f^t_{{\bf i}} ) }:\; k=0, 1,\ldots, d\right\} ,
\end{equation}
where   $f^t_{{\bf i}}:=f^t_{i_1}\circ \cdots \circ f^t_{i_n}$ for ${\bf i}=i_1\ldots i_n$,  $f^t_\varepsilon$ denotes the identity map on $\R^d$, and $\phi^s(\cdot)$ stands for the singular value function (see \eqref{e-singular} for the definition).

\begin{de}
\label{de-1.2}
{\rm
Let $\eta$ be a locally finite Borel  measure on ${\Omega}$. We say that the family ${\mathcal F}^t$, $t\in {\Omega}$, satisfies {\color{red} a} {\it generalized transversality condition} (GTC) with respect to $\eta$ if there exist $\delta_0>0$ and a function $\psi: (0,\delta_0)\to [0,\infty)$ with $\lim_{\delta\to 0}\psi(\delta)=0$ such that  the following statement holds:  for every $t_0\in {\Omega}$ and every $0<\delta<\delta_0$, there exists a constant $C=C(t_0, \delta)>0$ such that for all distinct ${\bf i}, {\bf j}\in \Sigma$ and $r>0$,
\begin{equation}
\label{e-f2}
\eta\left\{t\in B(t_0, \delta):  \; |\Pi^t({\bf i})-\Pi^t({\bf j})  |< r\right\}\leq C e^{|{\bf i}\wedge {\bf j}| \psi(\delta)} Z^{t_0}_{{\bf i}\wedge {\bf j}}(r) ,
\end{equation}
where $B(t_0,\delta)$ denotes the closed ball in ${\Omega}$ of radius $\delta$ centered at $t$,  $\bi\wedge \bj$ denotes the common initial segment of $\bi$ and $\bj$, and $|\bi\wedge \bj|$ is the length of the word $\bi\wedge \bj$.
}
\end{de}

The introduction of the GTC is inspired by the work of Jordan, Pollicott and Simon \cite{JordanPollicottSimon07} who defined the self-affine transversality condition for certain translational families of affine IFSs. The new feature here is that the upper bound term in the right-hand side of \eqref{e-f2} depends upon $t_0$, $\delta$ and $|{\bf i}\wedge {\bf j}|$, whilst in the setting of \cite{JordanPollicottSimon07} the corresponding upper bound term is independent of these parameters and  is determined by the linear parts of one pre-given affine IFS.


For $t\in {\Omega}$ and a $\sigma$-invariant measure $\mu$ on $\Sigma$, we write
\begin{equation}
\label{e-de1.6}
d(t):=\dim_S(\mathcal F^t),\quad d_\mu(t):=\dim_{L,\mathcal F^t}\mu
\end{equation}
for the singularity dimension of $\mathcal F^t$ and  the Lyapunov dimension of $\mu$ with respect to $\mathcal F^t$, respectively; see Definitions \ref{de-1}-\ref{de-2}. For $E\subset \R^d$, let $\dim_HE$ denote the Hausdorff dimension of $E$, and let  $\overline{\dim}_BE, \;\underline{\dim}_BE$ denote the upper and lower  box-counting dimensions of $E$, respectively (cf. \cite{Falconer2003}). When $\overline{\dim}_BE=\underline{\dim}_BE$, the common value is said to be the box-counting dimension of $E$ and is denoted by $\dim_BE$.

The first result of the present paper is the following.

\begin{thm}
\label{thm-f1.1}
Let ${\mathcal F}^t=\{f_1^t,\ldots, f_\ell^t\}$, $t\in {\Omega}$, be a parametrized family of $C^1$ IFSs defined on a common compact subset $Z$ of $\R^d$, such that the conditions (C1)-(C2) hold. Let $\eta$ be a locally finite Borel  measure on ${\Omega}$. Assume that $({\mathcal F}^t)_{t\in {\Omega}}$ satisfies the GTC with respect to $\eta$. Then the following properties hold.
\begin{itemize}
\item[(i)] Let  $\mu$  be a $\sigma$-invariant ergodic measure  on $\Sigma$.  For $\eta$-a.e.~$t\in {\Omega}$,  $\Pi^t_*\mu$ is exact dimensional and
$$
 \dim_H \Pi^t_*\mu=\min\{d, d_\mu(t)\}.
$$
Moreover, $\Pi^t_*\mu\ll {\mathcal L}_d$ for $\eta$-a.e.~$t\in \{t'\in {\Omega}: \; d_\mu(t')>d\}$, where ${\mathcal L}_d$ denotes the Lebesgue measure on $\R^d$.
\item[(ii)] For $\eta$-a.e.~$t\in {\Omega}$,
$$
\dim_H K^t=\dim_BK^t=\min\{d, d(t)\}.
$$
Moreover, ${\mathcal L}_d(K^t)>0$ for $\eta$-a.e.~$t\in \{t'\in {\Omega}: \; d(t')>d\}$.
\end{itemize}
\end{thm}

The above theorem is a nonlinear analogue of the results of Jordan, Pollicott and Simon \cite[Theorems 4.2-4.3]{JordanPollicottSimon07} for affine IFSs.  We emphasize that  in the nonlinear case, the singularity and Lyapunov dimensions depend on the parameter $t$,   whilst in the affine case, the corresponding quantities are constant. This is a  key difference between the affine case and the nonlinear case. We remark that Theorem \ref{thm-f1.1} also extends and generalizes the corresponding results of Simon, Solomyak and Urba\'{n}ski (\cite[Theorem 3.1]{SSU2001a}, \cite[Theorem 2.3]{SSU2001b}) for $C^{1+\alpha}$ conformal IFSs on $\R$.

Now a natural question arises how to verify the GTC for a parametrized family of $C^1$ IFSs. In what follows we investigate this question for certain translational families of $C^1$ IFSs.

First we introduce some definitions.

\begin{de}
{\rm Let $\mathcal F=\{f_i\}_{i=1}^\ell$ be a $C^1$ IFS on a compact set $Z\subset \R^d$ such  that $f_i(Z)\subset {\rm int}(Z)$ for each $i$.  Set
\begin{equation}
\label{e-translation}
f_i^{\pmb{\mathfrak{t}}}:=f_i+\mathbf{t}_i, \quad i=1,\ldots, \ell,
\end{equation}
 where $\pmb{\mathfrak{t}}=(\mathbf{t}_1,\ldots, \mathbf{t}_\ell)\in \R^{\ell d}$ with ${\bf t}_i\in \R^d$.  By continuity, there is a small $r_0>0$  such that $f_i^{\pmb{\mathfrak{t}}}(Z)\subset {\rm int}(Z)$ for every $\pmb{\mathfrak{t}}$ with $|\pmb{\mathfrak{t}}|<r_0$ and   every $i$, where $|\cdot|$ is the Euclidean norm.  Set $\mathcal F^{\pmb{\mathfrak{t}}}=\{f_i^{\pmb{\mathfrak{t}}}\}_{i=1}^\ell$ for each $\pmb{\mathfrak{t}}$ with $|\pmb{\mathfrak{t}}|<r_0$.
 We call  $(\mathcal F^{\pmb{\mathfrak{t}}})_{\pmb{\mathfrak{t}}\in \Delta}$, where $\Delta:=\{\pmb{\mathfrak{s}}\in \R^{\ell d}:\; |\pmb{\mathfrak{s}}|<r_0\}$, a  {\it translational family of IFSs generated by $\mathcal F$}.}
\end{de}

\begin{de}
{\rm
Let $\mathcal F=\{f_i\}_{i=1}^\ell$ be a $C^1$ IFS on a compact set $Z\subset \R^d$. We say that $\mathcal F$ is {\it dominated lower triangular},  if for each $z\in Z$ and $i\in \{1,\ldots, \ell\}$, the Jacobian  $D_zf_i$ of $f_i$ at $z$ is a lower triangular matrix such that
    $$
    |(D_zf_i)_{11}|\geq |(D_zf_i)_{22}|\geq \cdots\geq  |(D_zf_i)_{dd}|.
    $$
    }
\end{de}

We remark that in the above definition, the condition for an IFS to be  dominated lower triangular is slightly weaker than that required in \cite{FalconerFraserLee2020, JurgaLee2020}.

\begin{de}
\label{de-1.5}
{\rm Let $\ell\in \N$ with $\ell\geq 2$. Assume  for $j=1,\ldots, n$,  $\mathcal F_j=\{f_{i,j}\}_{i=1}^\ell$ is an IFS on $Z_j\subset \R^{q_j}$. Let $\mathcal F=\{f_i\}_{i=1}^\ell$ be an IFS on $Z_1\times\cdots\times Z_n\subset \R^{q_1}\times\cdots\times \R^{q_n}$ given by
$$
f_i(x_1,\ldots, x_n)=(f_{i,1}(x_1),\ldots, f_{i,n}(x_n)), \quad i=1,\ldots, \ell, \; x_k\in Z_k \mbox{ for }1\leq k\leq n.
$$
We say that $\mathcal F$ is the {\it direct product} of $\mathcal F_1,\ldots, \mathcal F_n$, and write $\mathcal F=\mathcal F_1\times\cdots\times \mathcal F_n$.
}
\end{de}

Now we are ready to state the second main result of the paper.

\begin{thm}
\label{thm-1.2}
Let $\mathcal F=\{f_i\}_{i=1}^\ell$ be a $C^1$ IFS on a compact set $Z\subset \R^d$ such that $f_i(Z)\subset {\rm int} (Z)$ for each $i$. Suppose either one of the following 3 conditions holds:
\begin{itemize}
\item[(i)]  $\mathcal F$ is  dominated lower triangular on $Z$ satisfying
\begin{equation}
\label{e-ratio}
\max_{i\neq j}\left(\sup_{y,z\in Z} \|D_yf_i\|+\|D_zf_j\|\right)<1,
 \end{equation}
 and $Z$ is convex.
 \item[(ii)]  $\mathcal F$ is  a $C^1$ conformal IFS on $Z$ satisfying \eqref{e-ratio}, and $Z$ is connected.
\item[(iii)] $\mathcal F=\mathcal F_1\times\cdots\times \mathcal F_n$, where for each $k\in \{1,\ldots, n\}$,  $\mathcal F_k$ is a $C^1$ IFS on a compact $Z_k\subset \R^{d_k}$ satisfying either (i) or (ii), in which $\mathcal F$ and $Z$ are replaced by $\mathcal F_k$ and $Z_k$ respectively.
\end{itemize}
Then there is a small $r_0>0$ such that the translational family $$\mathcal F^{\tt}=\{f_i+\mathbf{t}_i\}_{i=1}^\ell,\quad
\pmb{\mathfrak{t}}=(\mathbf{t}_1,\ldots, \mathbf{t}_\ell)\in \Delta:=\{\pmb{\mathfrak{s}}\in \R^{\ell d}:\; |\pmb{\mathfrak{s}}|<r_0\},$$
satisfies the GTC with respect to the Lebesgue measure $\mathcal L_{\ell d}$ on $\Delta$. As a consequence, the conclusions of Theorem \ref{thm-f1.1} hold for the family $(\mathcal F^{\tt})_{\tt\in \Delta}$.
\end{thm}

The above theorem is a (partial) non-linear  extension of  the corresponding results in \cite{Falconer88, solomyak1998measure, JordanPollicottSimon07} for affine IFSs. Recall that in the case when $\mathcal F=\{f_i(x)=A_ix+a_i\}_{i=1}^\ell$ is an affine IFS on $\R^d$, under the assumption that
\begin{equation}
\label{e-Fbound}
\max_{1\leq i\leq \ell}\|A_i\|<1/3,
\end{equation}
 Falconer \cite{Falconer88} proved that the dimension of the attractor of $\mathcal F^{\tt}=\{f_i+{\bf t}_i\}_{i=1}^\ell$ is equal to its affinity dimension for ${\mathcal L}_{\ell d}$-a.e.~$\tt=({\bf t}_1,\ldots, {\bf t}_\ell)\in \R^{\ell d}$.  Later Solomyak \cite{solomyak1998measure}  pointed out that the bound $1/3$ in \eqref{e-Fbound} can be replaced by $1/2$. By an observation of Edgar \cite{edgar1992fractal}, $1/2$ is optimal. Under the same assumption that
 \begin{equation}
\label{e-Fbound1}
\max_{1\leq i\leq \ell}\|A_i\|<1/2,
\end{equation}
 Jordan, Pollicott and Simon \cite{JordanPollicottSimon07} showed that the translational family $(\mathcal F^{\tt})_{\tt\in \R^{\ell d}}$ satisfies the self-affine transversality condition. It was pointed out in \cite[Theorem 9.1.2]{BSS2021} that the assumption \eqref{e-Fbound1} can by further replaced by a slightly more general condition $\max_{i\neq j}(\|A_i\|+\|A_j\|)<1$.

 We remark that Theorem \ref{thm-1.2} also extends the results of Simon, Solomyak and Urba\'{n}ski (\cite[Proposition 7.1]{SSU2001a}, \cite[Corollary 7.3]{SSU2001b}) for  translational families of $C^{1+\alpha}$ conformal IFSs on $\R$.
It is worth pointing out that for every $C^1$ conformal IFS satisfying the open set condition (or  $C^1$ conformal expanding map), the dimension of its attractor (or repeller) satisfies the Bowen-Ruelle formula, and is equal to the singularity dimension; meanwhile the dimension of ergodic invariant measures on the attractor (repeller) is given by the Lyapunov dimension  (see \cite{Bowen1979,  Ruelle1982, GatzourasPeres1997, Patzschke1997}).

The paper is organized as follows. In Section~\ref{S-2}, we give some preliminaries, including the variational principle for sub-additive topological pressure, the definitions and properties of singularity dimension and Lyapunov dimension. In Section~\ref{S-f1.1}, we prove Theorem \ref{thm-f1.1}. The proof of Theorem \ref{thm-1.2} is rather long and will be given in Sections \ref{S-triangular}-\ref{S-thm-1.2}, where we divide the whole proof into 3 different parts, by considering the conditions (i)-(iii) in Theorem \ref{thm-1.2} separately.

\section{Preliminaries}
\label{S-2}
\subsection{Variational principle for sub-additive pressure}
\label{subsec:TF}
In order to define the singularity and Lyapunov dimensions,  we require some elements from the sub-additive thermodynamic formalism.


Let $(\Sigma,\sigma)$ be the one-sided full shift over the alphabet $\{1,\ldots, \ell\}$. That is, $\Sigma=\{1,\ldots, \ell\}^{\Bbb N}$, which is endowed with the product topology, and $\sigma:\Sigma\to \Sigma$ is the left shift defined by $(x_i)_{i=1}^\infty\mapsto (x_{i+1})_{i=1}^\infty$.  Write $\Sigma_n=\{1,\ldots, \ell\}^n$ for $n\geq 0$, with  the convention $\Sigma_0=\{\varepsilon\}$, where $\varepsilon$ stands for the empty word. Set $\Sigma_*=\bigcup_{n=0}^\infty \Sigma_n$.  For $x=(x_i)_{i=1}^\infty\in \Sigma$ and $n\in {\Bbb N}$, write $x|n=x_1\ldots x_n$.

Let $C(\Sigma)$ denote the set of real-valued continuous functions on $\Sigma$. Let $\mathcal{G}=\left\{g_n\right\}_{n=1}^{\infty }$ be a {\it sub-additive potential} on $\Sigma$, that is,  $g_n\in C(\Sigma)$ for all $n\geq 1$ such that
  \begin{equation}\label{090}
    g_{m+n}(x) \leq g_{n}(x)+g_{m}(\sigma^nx) \quad\mbox{ for  all }  x\in \Sigma \mbox { and }n,m\in \N.
  \end{equation}
 The {\it topological pressure of $\mathcal G$} is defined by
 \begin{equation}
 \label{e-Falconer}
 P(\Sigma,\sigma, \mathcal G)= \lim_{n\to\infty} \frac{1}{n}\log\left( \sum_{I\in \Sigma_n} \sup_{x\in [I]} \exp(g_n(x)) \right),
\end{equation}
where $[I]:=\{x\in \Sigma:\; x|n=I\}$ for $I\in \Sigma_n$.
The limit can be seen to exist by using a standard sub-additivity argument.

If the potential $\mathcal G$ is additive, i.e. $g_n=\sum_{k=0}^{n-1}g\circ \sigma^k$ for some  $g\in C(\Sigma)$, then $P(\Sigma, \sigma, \mathcal G)$ recovers the classical topological pressure  $P(\Sigma, \sigma, g)$ of $g$ (see e.g.~\cite{Walters82}).

Let  $\mathcal{M}(\Sigma,\sigma)$ denote the set of  $\sigma$-invariant Borel probability measures on $\Sigma$.
For $\mu\in \mathcal{M}(\Sigma,\sigma)$, let $h_\mu(\sigma)$  denote the measure-theoretic entropy of $\mu$ (cf. \cite{Walters82}). Moreover, for  $\mu \in\mathcal M(\Sigma, \sigma)$, by sub-additivity,
 \begin{equation}
 \label{e-N1}
 \mathcal{G}_*(\mu):=\lim_{n\to\infty} \frac{1}{n} \int g_n d\mu=\inf_n  \frac{1}{n} \int g_n d\mu\in [-\infty,\infty).
 \end{equation}
 See e.g.~\cite[Theorem 10.1]{Walters82}.  We call  $\mathcal{G}_*(\mu)$  the {\it Lyapunov exponent} of $\mathcal{G}$ with respect to $\mu$.

The following variational principle for the topological pressure of sub-additive potentials generalizes the classical variational principle  for additive potentials (\cite{Ruelle1973, Walters1975}):
\begin{thm} [\cite{CFH08}]
\label{thm:sub-additive-VP}  
Let $\mathcal{G}=\left\{g_n\right\}_{n=1}^{\infty }$ be a sub-additive potential on  $(\Sigma,\sigma)$. Then
  \begin{equation}\label{variational-principle}
    P(\Sigma, \sigma,\mathcal{G})=
    \sup\left\{h_\mu(\sigma)+\mathcal{G}_*(\mu):\;
    \mu\in\mathcal{M}(\Sigma,\sigma)
    \right\}.
  \end{equation}
\end{thm}
Although in \cite{CFH08} this is proved for sub-additive potentials on an arbitrary continuous dynamical system on a compact space, we state it only for  shift spaces. Particular cases of the above result, under stronger assumptions on the potentials, were previously obtained by many authors, see for example \cite{Falconer88b, FengLau2002, Feng2004, Kaenmaki04, Mummert06, Barreira10} and references therein.

Measures that achieve the supremum in \eqref{variational-principle}  are called  {\em  equilibrium measures} for the potential ${\mathcal G}$. There exists at least one ergodic equilibrium measure;  see e.g.~\cite[Proposition 3.5]{Feng11} and the remark there.

\subsection{Singularity dimension and Lyapunov dimension with respect to  $C^1$ IFSs}
\label{S-2.3}
In this subsection, we define the singularity  and Lyapunov dimensions with respect to $C^1$ IFSs.

Let $\mathcal F=\{f_i\}_{i=1}^\ell$ be a $C^1$ IFS on a compact set $Z\subset \R^d$ and let $K$ denote the attractor of $\mathcal F$ (cf. Section~\ref{S-1}). Let $(\Sigma,\sigma)$ be the one-sided full shift over the alphabet $\{1,\ldots, \ell\}$ and  let $\Pi: \Sigma\to K$ be the coding map defined as in \eqref{e-code}.

Let $\R^{d\times d}$ denote the collection of $d\times d$ real matrices.  For $T\in \R^{d\times d}$, let $$\alpha_1(T)\geq\cdots\geq \alpha_d(T)$$ denote the  singular values of $T$. Following \cite{Falconer88},  for $s\geq 0$ we define the {\it singular value function} $\phi^s:\; \R^{d \times d}\to [0,\infty)$ as
\begin{equation}
\label{e-singular}
\phi^s(T)=\left\{
\begin{array}
{ll}
\alpha_1(T)\cdots \alpha_k(T) \alpha_{k+1}^{s-k}(T) & \mbox{ if }0\leq s< d,\\
\det(T)^{s/d} & \mbox{ if } s\geq d,
\end{array}
\right.
\end{equation}
where $k=[s]$ is the integral part of $s$. Here we make the convention that $0^0=1$. The following result on $\phi^s$ is well known; see e.g. \cite{Falconer88}.

\begin{lem}
\label{lem-inequality}
\begin{itemize}
\item[(i)] $\phi^s(ST)\leq \phi^s(S)\phi^s(T)$ for all $S, T\in \R^{d\times d}$ and $s\geq 0$.
\item[(ii)] $\phi^{s+t}(T)\leq \phi^s(T)\|T\|^t$ for all $T\in \R^{d\times d}$, $s,t\geq 0$.
\end{itemize}
\end{lem}

For a differentiable mapping $f:\; U\subset \R^d\to \R^d$, let $D_zf$ denote the differential of $f$ at $z\in U$. Sometimes we also write $f'(z)$ for $D_zf$, and also call $D_zf$ the Jacobian matrix of $f$ at $z$. Below we introduce the concepts of singularity and Lyapunov dimensions.

\begin{de}
\label{de-1}
{\rm
 The {\it singularity dimension of  $\mathcal F=\{f_i\}_{i=1}^\ell$}, written as  $\dim_S\mathcal F$, is the unique non-negative value $s$ for which
$$
P(\Sigma,\sigma, \mathcal G^s)=0,
$$
where $\mathcal G^s=\{g_n^s\}_{n=1}^\infty$ is the sub-additive potential on $\Sigma$ defined by
\begin{equation}
\label{e-gn}
g_n^s(x)=\log \phi^s(D_{\Pi\sigma^n x}f_{x|n}), \quad x\in \Sigma,
\end{equation}
with $f_{x|n}:=f_{x_1}\circ \cdots\circ f_{x_n}$ for $x=(x_n)_{n=1}^\infty$.
}
\end{de}

\begin{de}
\label{de-2}
{\rm
Let $\mu$ be a $\sigma$-invariant Borel probability measure on $\Sigma$. The {\it Lyapunov dimension of $\mu$ with respect to $\mathcal F=\{f_i\}_{i=1}^\ell$}, written  as $\dim_{\rm L, \mathcal F} \mu$, is the unique non-negative value $s$ for which
$$
h_\mu(\sigma)+ \mathcal G^s_*(\mu)=0,
$$
where $\mathcal G^s=\{g^s_n\}_{n=1}^\infty$ is defined as in  \eqref{e-gn} and $\mathcal G^s_*(\mu):=\lim_{n\to \infty} \frac{1}{n}\int g^s_n\; d\mu$.
}
\end{de}


\begin{rem}
{\rm
\begin{itemize}
\item[(i)] It is not hard to show that there exist $a<b<0$ such that
$$
nsa\leq g_n^s(x)\leq ns b,\qquad g_n^{s+t}(x)\leq  g_n^s(x)+ntb
$$
for all $x\in \Sigma$, $n\in \N$ and $s,t\geq 0$, where $g_n^s(x)$ is defined as in \eqref{e-gn}. The existence and uniqueness of $s$ in Definitions \ref{de-1}-\ref{de-2} just follow from this fact.

\item[(ii)] The concept of singularity dimension was first introduced by Falconer \cite{Falconer88, Falconer94}; see also \cite{KaenmakiVilppolainen10}. It is also called {\it affinity dimension} in the case when the IFS $\{f_i\}_{i=1}^\ell$ is affine, that is, each map $f_i$ is affine.
\item[(iii)] The definition of Lyapunov dimension of invariant measures with respect to an IFS  presented above was adopted from \cite{JordanPollicott08}. It is a  generalization of that given in \cite{JordanPollicottSimon07} for affine IFSs.
\end{itemize}
}
\end{rem}

The following result describes  the relation between the singularity dimension and the Lyapunov dimension.
\begin{lem}
\label{lem-var}
Let $\mathcal F=\{f_i\}_{i=1}^\ell$ be a $C^1$ IFS on a compact subset $Z$ of $\R^d$. Suppose $\theta\in (0,1)$ is a common Lipschitz constant for $f_1,\ldots, f_\ell$. That is,
$$
|f_i(x)-f_i(y)|\leq \theta|x-y| \quad \mbox{ for all }1\leq i\leq \ell,\; x,y\in Z.
$$
Then the following properties hold.
\begin{itemize}
\item[(i)] $\dim_S\mathcal F=\sup\{\dim_{L,\mathcal F}\mu:\; \mu\in \mathcal M(\Sigma,\sigma)\}$.  The supremum is attained by at least one  ergodic measure.
\item[(ii)]
 $\dim_S\mathcal F\leq \frac{\log \ell}{\log (1/\theta)}$.
\end{itemize}
\end{lem}
\begin{proof}
Since $\theta$ is a common Lipschitz constant for $f_1,\ldots, f_\ell$, $\|D_zf_i\|\leq \theta$ for each
$1\leq i\leq \ell$ and $z\in Z$.  It follows from Lemma \ref{lem-inequality}(ii) that for $s_2>s_1\geq 0$,
$$
\phi^{s_2}(D_{\Pi\sigma^n x} f_{x|n})\leq \phi^{s_1}(D_{\Pi\sigma^n x} f_{x|n}) \theta^{n(s_2-s_1)} \quad\mbox{ for all }x\in \Sigma,\; n\in \N,
$$
from which we see that $$P(\Sigma,\sigma,\mathcal G^{s_2})\leq P(\Sigma,\sigma,\mathcal G^{s_1})-(s_2-s_1)\log (1/\theta).
$$
Hence $P(\Sigma, \sigma, \mathcal G^s)$ is strictly decreasing in $s$.

Now let  $\mu\in \mathcal M(\Sigma,\sigma)$. Write $s=\dim_{L,\mathcal F}\mu$. Then $h_\mu(\sigma)+\mathcal G^s_*(\mu)=0$. Applying  Theorem \ref{thm:sub-additive-VP} to  the sub-additive potential $\mathcal G^s$ yields that
$P(\Sigma, \sigma, \mathcal G^s)\geq 0$.  Hence $\dim_S\mathcal F\geq s=\dim_{L,\mathcal F}\mu$.  It follows that
$$\dim_S\mathcal F\geq \sup\{\dim_{L,\mathcal F}\mu:\; \mu\in \mathcal M(\Sigma,\sigma)\}.$$
To show that  equality holds, write $s'=\dim_S\mathcal F$. Let $\nu$ be an ergodic equilibrium measure for the potential $\mathcal G^{s'}$.  Then
$$0=P(\Sigma, \sigma, \mathcal G^{s'})=h_\nu(\sigma)+\mathcal G^{s'}_*(\nu),$$
which implies that  $\dim_{L, \mathcal F}\nu=s'$. That is, $\dim_{L, \mathcal F}\nu=\dim_S\mathcal F$. This completes the proof of (i).

To see (ii), notice that $\phi^{s'}(D_{\Pi\sigma^n x} f_{x|n})\leq \theta^{ns'}$ for all $x\in \Sigma$ and $n\in \N$.
 It follows from the definition of $P(\Sigma, \sigma, \mathcal G^{s'})$ that
$$
0=P(\Sigma, \sigma, \mathcal G^{s'})\leq \lim_{n\to \infty} \frac{1}{n} \log (\ell^n \theta^{ns'})=\log \ell +s'\log \theta,
$$
from which we obtain $s'\leq \frac{\log \ell}{\log (1/\theta)}$. This completes the proof of (ii).
\end{proof}

For a $\sigma$-invariant ergodic measure $\mu$ on $\Sigma$,  let $\Pi_*\mu$ denote the push-forward of $\mu$ by $\Pi$. In the following we present the main result obtained in the first part \cite{FengSimon2020} of our study on the dimension of $C^1$  iterated function systems:  the upper box-counting dimension of the attractor of $\mathcal F$ is bounded above by the singularity dimension of $\mathcal F$, whilst the upper packing dimension of $\Pi_*\mu$ is bounded above by the Lyapunov dimension of $\mu$.

\begin{thm}[\cite{FengSimon2020}]
\label{thm-FS}
Let $\mathcal F=\{f_i\}_{i=1}^\ell$ be a $C^1$ IFS  with attractor $K$, and let $\mu$ be a $\sigma$-invariant ergodic measure on $\Sigma$. Then the following properties hold.
\begin{itemize}
\item[(i)]  $\overline{\dim}_B K\leq \dim_S \mathcal F$.
\item[(ii)] $\overline{\dim}_P \Pi_*\mu\leq \dim_{L,\mathcal F}(\mu)$.
\end{itemize}
\end{thm}

\section{The proof of Theorem \ref{thm-f1.1}}
\label{S-f1.1}

In this section, we prove Theorem \ref{thm-f1.1}.
A key part of the proof is the following proposition.

\begin{pro}
\label{pro-f1.1}
Assume that $\mathcal F^t$, $t\in {\Omega}$, satisfies the GTC with respect to a locally finite Borel measure $\eta$ on ${\Omega}$.
Let $\mu$ be a $\sigma$-invariant ergodic measure on $\Sigma$. Let $t_0\in {\Omega}$ and $0<\delta<\delta_0$, where $\delta_0$ is given as in Definition \ref{de-1.2}, Then the following properties hold.
\begin{itemize}
\item[(i)] For $\eta$-a.e.~$t\in B(t_0,\delta)$,
\begin{equation}
\label{e-f5}
\underline{\dim}_H\Pi^t_*\mu\geq \min\{d, d_\mu(t_0)\}-\frac{\psi(\delta)}{\log (1/\theta)},
\end{equation}
where  $\psi(\cdot)$ is given as in Definition \ref{de-1.2}, and $\theta$ is given as in \eqref{e-f1}.
\item[(ii)] If  $d_\mu(t_0)>d+\frac{\psi(\delta)}{\log (1/\theta)}$,   then $\Pi^t_*\mu\ll \mathcal L_d$ for $\eta$-a.e.~$t\in B(t_0,\delta)$.
\end{itemize}
\end{pro}

The proof of the above proposition is adapted from  an argument used  in \cite[Propositions 4.3-4.4]{JordanPollicottSimon07}. For the reader's convenience, we include a full proof.  We begin with the following.

\begin{lem}
Assume that $(\mathcal F^t)_{t\in {\Omega}}$ satisfies the GTC with respect to a locally finite Borel measure $\eta$ on ${\Omega}$. Let $s$ be non-integral with $0<s<d$. Let $t_0\in {\Omega}$ and $0<\delta<\delta_0$, where  $\delta_0$ is given as in Definition \ref{de-1.2}. Then there exists a number
$c>0$, dependent on $s$ and $\delta$, such that for all distinct $\bi,\bj\in \Sigma$,
\begin{equation}
\label{e-f3}
\int_{B(t_0,\delta)} |\Pi^t({\bf i})-\Pi^t(\bj) |^{-s} \;d\eta(t)\leq c e^{|\bi\wedge \bj|\psi(\delta)} \left( \max_{x\in \Sigma} \phi^s (D_{\Pi^{t_0} x} f^{t_0}_{\bi\wedge \bj})\right)^{-1},
\end{equation}
where  $\psi(\cdot)$ is given as in Definition \ref{de-1.2}.
\end{lem}
\begin{proof}
Take $y\in \Sigma$ so that
$
\phi^s (D_{\Pi^{t_0} y} f^{t_0}_{\bi\wedge \bj})= \max_{x\in \Sigma} \phi^s (D_{\Pi^{t_0} x} f^t_{\bi\wedge \bj}).
$
Let $k$ be the unique integer such that $s\in (k, k+1)$. Clearly $k\in \{0,1,\ldots, d-1\}$. For convenience, write
$$
a:=\phi^k (D_{\Pi^{t_0} y} f^{t_0}_{\bi\wedge \bj}),\quad b:=\phi^{k+1}(D_{\Pi^{t_0} y} f^{t_0}_{\bi\wedge \bj}),
$$
where $\phi^s(\cdot)$ stands for the singular value function (see \eqref{e-singular} for the definition).
A direct check shows that
\begin{equation}
\label{e-f4}
\phi^s (D_{\Pi^{t_0} y} f^t_{\bi\wedge \bj})=a^{k+1-s}b^{s-k}.
\end{equation}
Observe that
\begin{align*}
\int_{B(t_0,\delta)} & |\Pi^t({\bf i})-\Pi^t(\bj) |^{-s} \;d\eta(t)\\
&=s\int_0^\infty r^{-s-1} \eta \{t\in B(t_0,\delta):\; |\Pi^t({\bf i})-\Pi^t({\bf j})  |< r\}\; dr\\
&\leq sCe^{|\bi\wedge \bj|\psi(\delta)}   \int_0^\infty r^{-s-1} Z^{t_0}_{\bi\wedge \bj}(r) \; dr \qquad (\mbox{by \eqref{e-f2}})\\
&\leq sCe^{|\bi\wedge \bj|\psi(\delta)}   \int_0^\infty r^{-s-1} \min\left\{ \frac{r^k}{a}, \frac{r^{k+1}}{b} \right\} \; dr \qquad (\mbox{by \eqref{e-f1.0}})\\
&\leq sCe^{|\bi\wedge \bj|\psi(\delta)}   \int_0^{b/a}  \frac{r^{k-s}}{b} \;dr+ \int_{b/a}^\infty  \frac{r^{k-s-1}}{a}\; dr\\
&=sCe^{|\bi\wedge \bj|\psi(\delta)} \left(\frac{1}{k+1-s}+\frac{1}{s-k}\right)a^{s-k-1}b^{k-s}\\
&=sC\left(\frac{1}{k+1-s}+\frac{1}{s-k}\right)e^{|\bi\wedge \bj|\psi(\delta)} (\phi^s (D_{\Pi^{t_0} y}f^{t_0}_{\bi\wedge \bj}))^{-1}\quad (\mbox{by \eqref{e-f4})}.
\end{align*}
This proves \eqref{e-f3} by setting $c=sC\left(\frac{1}{k+1-s}+\frac{1}{s-k}\right)$.
\end{proof}

\begin{proof}[Proof of Proposition \ref{pro-f1.1}] Fix $t_0\in {\Omega}$ and $\delta\in (0,\delta_0)$. We first prove part (i).
  Let $\epsilon>0$ and let $s$ be non-integral so that
\begin{equation}
\label{e-ff0}
0<s<\min\{d, d_\mu(t)\}-\frac{\psi(\delta)}{\log (1/\theta)}-2\epsilon.
\end{equation}
 To show that \eqref{e-f5} holds for $\eta$-a.e.~$t\in B(t_0,\delta)$, it suffices to show that
\begin{equation}
\label{e-ff0'}
\underline{\dim}_H\Pi^t_*\mu\geq s \quad \mbox{for $\eta$-a.e.~}t\in B(t_0,\delta).
\end{equation}

For this purpose,  we write
\begin{equation}
\label{e-ff1}
\varphi^s(I)=\max_{x\in \Sigma} \phi^s(D_{\Pi^{t_0}_x} f_I^{t_0}),\quad  I\in \Sigma_*.
\end{equation}
We first prove that
\begin{equation}
\label{e-ff2}
\lim_{n\to \infty}\frac{\mu([\bi|n])}{\varphi^s(\bi|n) \exp(-n \psi(\delta))\theta^{n\epsilon}}=0\quad \mbox{ for $\mu$-a.e.~}\bi\in \Sigma.
\end{equation}
To see this, according to the definition of $d_\mu(t_0)$ (cf. \eqref{e-de1.6} and Definition \ref{de-2}),
\begin{equation}
\label{e-ff3}
h_\mu(\sigma)+\lim_{n\to \infty}\frac{1}{n}\int \log \phi^{d_\mu(t_0)} (D_{\Pi^{t_0}\sigma^n\bi} f^{t_0}_{\bi|n})\; d\mu(\bi)=0.
\end{equation}
It follows from \eqref{e-ff3}, the Shannon-McMillan-Breiman theorem and Kingman's sub-additive ergodic theorem (see \cite[p.~93 and p.~231]{Walters82}) that
\begin{equation}
\label{e-ff4}
\lim_{n\to \infty}\frac{1}{n}\log \frac{\mu([\bi|n])}{\phi^{d_\mu(t_0)} (D_{\Pi^{t_0}\sigma^n\bi} f^{t_0}_{\bi|n})}=0 \quad \mbox{ for $\mu$-a.e.~}\bi\in \Sigma.
\end{equation}
Observe that for each $\bi\in \Sigma$ and $n\in \N$,
\begin{align*}
\phi^{d_\mu(t_0)} (D_{\Pi^{t_0}\sigma^n\bi} f^{t_0}_{\bi|n})
 &\leq \phi^{s} (D_{\Pi^{t_0}\sigma^n\bi} f^{t_0}_{\bi|n}) \|D_{\Pi^{t_0}\sigma^n\bi} f^{t_0}_{\bi|n}\|^{d_\mu(t_0)-s}
 \qquad \mbox{(by Lemma \ref{lem-inequality}(ii))}\\
 &\leq \varphi^{s} (\bi|n) \theta^{n(d_\mu(t_0)-s)}\qquad\qquad \mbox{(by \eqref{e-ff1} and \eqref{e-f1})}\\
 &\leq  \varphi^s(\bi|n) \exp(-n \psi(\delta))\theta^{2n\epsilon} \qquad \mbox{(by \eqref{e-ff0})}.
 \end{align*}
 Combining the above inequality with \eqref{e-ff4} yields \eqref{e-ff2}.

 By \eqref{e-ff2},  we
may find a countable disjoint collection  of Borel subsets $E_j$ of $\Sigma$ with $\mu(\Sigma\backslash \bigcup_{j=1}^\infty E_j)=0$ and numbers $c_j>0$ such that
\begin{equation}
\label{e-ff5}
\mu_j([\bi|n])\leq c_j  \varphi^s(\bi|n) \exp(-n \psi(\delta))\theta^{n\epsilon} \quad \mbox{ for all }\bi\in \Sigma,\; n\in \N,
\end{equation}
where $\mu_j$ stands for the restriction of $\mu$ to $E_j$ defined by $\mu_j(A)=\mu(E_j\cap A)$. Clearly,
\begin{equation}
\label{e-ff6}
\mu=\sum_{j=1}^\infty \mu_j.
\end{equation}
Hence,  to prove \eqref{e-ff0'} it suffices to show that for each $j$,
\begin{equation}
\label{e-ff0''}
\underline{\dim}_H\Pi^t_*\mu_j\geq s \quad \mbox{for $\eta$-a.e.~}t\in B(t_0,\delta).
\end{equation}
By the potential theoretic characterization of the Hausdorff dimension (see e.g. \cite[Theorem
4.13]{Falconer2003}), it is enough to show that for each $j$ and $\eta$-a.e.~$t\in B(t_0,\delta)$, $\Pi^t_*\mu_j$ has finite $s$-energy:
$$
I_s(\Pi^t_*\mu_j):=\iint  \frac{ d \Pi^t_*\mu_j (x)d \Pi^t_*\mu_j (y)}{|x-y|^{s}}<\infty.
$$

 Integrating over $B(t_0,\delta)$ with respect to $\eta$ and using Fubini's theorem,
 \begin{align*}
 \int_{B(t_0,\delta)}I_s(\Pi^t_*\mu_j)\; d\eta(t)&=\int_{B(t_0,\delta)} \iint  \frac{ d \Pi^t_*\mu_j (x)d \Pi^t_*\mu_j (y)}{|x-y|^{s}} d\eta(t)\\
 &=\int_{B(t_0,\delta)} \iint  \frac{ d\mu_j (\bi)d \mu_j (\bj)}{|\Pi^t(\bi)-\Pi^t(\bj)|^{s}} d\eta(t)\\
 &= \iint \int_{B(t_0,\delta)} \frac{d\eta(t)}{|\Pi^t(\bi)-\Pi^t(\bj)|^{s}} \ d\mu_j (\bi)d \mu_j (\bj)\\
 &\leq \iint c e^{|\bi\wedge \bj|\psi(\delta)}(\varphi^s(\bi\wedge \bj))^{-1}  d\mu_j (\bi)d \mu_j (\bj)\quad \mbox{(by \eqref{e-f3}, \eqref{e-ff1})}\\
 &\leq  c\int \sum_{n=0}^\infty e^{n\psi(\delta)} (\varphi^s(\bj|n))^{-1} \mu_j([\bj|n])d \mu_j (\bj)\\
 &\leq  cc_j\int \sum_{n=0}^\infty \theta^{n\epsilon} d \mu_j (\bj)\quad \mbox{(by \eqref{e-ff5})}\\
 &<\infty.
  \end{align*}
It follows that $I_s(\Pi^t_*\mu_j)<\infty$ for $\eta$-a.e.~$t\in B(t_0,\delta)$. This completes the proof of part (i).

Next we prove part (ii).  Take a small $\epsilon>0$ so that
\begin{equation}
\label{e-to1}
d_\mu(t_0)>d+\frac{\psi(\delta)}{\log (1/\theta)}+2\epsilon.
\end{equation}
Then  for every $\bi\in \Sigma$ and $n\in \N$,
\begin{align*}
\phi^{d_\mu(t_0)} (D_{\Pi^{t_0}\sigma^n\bi} f^{t_0}_{\bi|n})
 &\leq \phi^{d} (D_{\Pi^{t_0}\sigma^n\bi} f^{t_0}_{\bi|n}) \|D_{\Pi^{t_0}\sigma^n\bi} f^{t_0}_{\bi|n}\|^{d_\mu(t_0)-d}
 \qquad \mbox{(by Lemma \ref{lem-inequality}(ii))}\\
 &\leq \varphi^{d} (\bi|n) \theta^{n(d_\mu(t_0)-d)}\qquad\qquad \mbox{(by \eqref{e-ff1} and \eqref{e-f1})}\\
 &\leq  \varphi^d(\bi|n) \exp(-n \psi(\delta))\theta^{2n\epsilon} \qquad \mbox{(by \eqref{e-to1})}.
 \end{align*}
 Combining the above inequality with \eqref{e-ff4} yields that
$$
\lim_{n\to \infty}\frac{\mu([\bi|n])}{\varphi^d(\bi|n) \exp(-n \psi(\delta))\theta^{n\epsilon}}=0\quad \mbox{ for $\mu$-a.e.~}\bi\in \Sigma.
$$
Hence there exist  finite positive measures $\nu_j$ and numbers $a_j$ ($j\geq 1)$ such that  $\mu=\sum_{j=1}^\infty \nu_j$ and
\begin{equation}
\label{e-ff10}
\nu_j([\bi|n])\leq a_j  \varphi^d(\bi|n) \exp(-n \psi(\delta))\theta^{n\epsilon} \quad \mbox{ for all }\bi\in \Sigma,\; n\in \N.
\end{equation}

Since $\mu=\sum_{j=1}^\infty \nu_j$, to show that $\Pi^t_*\mu\ll \mathcal L_d$ for $\eta$-a.e.~$t\in B(t_0,\delta)$,  it suffices to show that for each $j$,
\begin{equation}
\label{e-ff11}
\Pi^t_*\nu_j\ll \mathcal L_d \quad \mbox{ for $\eta$-a.e.~}t\in B(t_0,\delta).
\end{equation}
To this end, fix $j$. We will follow a standard approach (introduced by Peres and Solomyak in \cite{PeresSolomyak1996}). In particular, it suffices to show
that
$$
I:=\int_{B(t_0,\delta)} \int \liminf_{r\to 0} \frac{\Pi^t_*\nu_j(B_{\R^d}(x,r)) }{r^d}\; d\Pi^t_*\nu_j(x) d\eta(t)
<\infty,
$$
here $B_{\R^d}(x,r)$ stands for the closed ball in $\R^d$ centered at $x$ of radius $r$.
Observe  that by \eqref{e-f1.0} and \eqref{e-ff1},
\begin{equation}
\label{e-ff12}
Z^{t_0}_{\pmb{\omega}}(r)\leq
\inf_{x\in \Sigma}  \frac{r^d}{ \phi^d (D_{\Pi^{t_0} x}f^t_{\pmb{\omega}} ) }=\frac{r^d}{\varphi^d(\pmb{\omega})}\quad \mbox { for }\pmb{\omega}\in \Sigma_*,\; r>0.
\end{equation}
Applying Fatou’s Lemma and Fubini’s Theorem,
\begin{align*}
I&\leq \liminf_{r\to 0}\frac{1}{r^d} \int_{B(t_0,\delta)} \int \Pi^t_*\nu_j(B_{\R^d}(x,r)) \; d\Pi^t_*\nu_j(x) d\eta(t)\\
&=\liminf_{r\to 0}\frac{1}{r^d} \int_{B(t_0,\delta)} \iint  {\bf 1}_{\{ (x,y): \; |x-y|\leq r\}}   \; d\Pi^t_*\nu_j(x)d\Pi^t_*\nu_j(y) d\eta(t)\\
&=\liminf_{r\to 0}\frac{1}{r^d} \int_{B(t_0,\delta)} \iint  {\bf 1}_{\{(\bi, \bj):\;  |\Pi^t(\bi)-\Pi^t(\bj)|\leq r\}} \; d\nu_j(\bi)d\nu_j(\bj) d\eta(t) \\
&=\liminf_{r\to 0}\frac{1}{r^d} \iint   \int {\bf 1}_{\{ t\in B(t_0,\delta):\; |\Pi^t(\bi)-\Pi^t(\bj)|\leq r\}} \; d\eta(t)d\nu_j(\bi)d\nu_j(\bj)  \\
&=\liminf_{r\to 0}\frac{1}{r^d} \iint   \eta\{t\in B(t_0,\delta):\; |\Pi^t(\bi)-\Pi^t(\bj)|\leq r\} \; d\nu_j(\bi)d\nu_j(\bj).
\end{align*}
By \eqref{e-f2} and  \eqref{e-ff12}, we obtain that
\begin{align*}
I &\leq \liminf_{r\to 0}\frac{1}{r^d} \iint   C e^{|{\bf i}\wedge {\bf j}| \psi(\delta)} Z_{{\bf i}\wedge {\bf j}}^{t_0}(r) \; d\nu_j(\bi)d\nu_j(\bj)\\
&\leq \iint   \frac{C e^{|{\bf i}\wedge {\bf j}| \psi(\delta)}}{\varphi^d(\bi\wedge \bj)} \; d\nu_j(\bi)d\nu_j(\bj)\\
&\leq \int \sum_{n=0}^\infty  \frac{C e^{n \psi(\delta)}}{\varphi^d( \bj|n)}\nu_j ([\bj |n])\; d\nu_j(\bj)\\
&\leq a_jC\int \sum_{n=0}^\infty  \theta^{n\epsilon}\; d\nu_j(\bj) \qquad \mbox{ (by \eqref{e-ff10})}\\
&<\infty,
\end{align*}
which completes the proof of part (ii).
\end{proof}

Now we are ready to prove Theorem \ref{thm-f1.1}.

\begin{proof}[Proof of Theorem \ref{thm-f1.1}(i)]

Let $\mu$ be a $\sigma$-invariant ergodic measure on $\Sigma$. We first show that for $\eta$-a.e.~$t\in {\Omega}$,
$\Pi^t_*\mu$ is exact dimensional with dimension equal to $\min\{d, d_\mu(t)\}$.  Recall that  $\overline{\dim}_P\Pi^t_*\mu\leq \min\{d, d_\mu(t)\}$ for each $t\in {\Omega}$ (see Theorem \ref{thm-FS}(ii)). Hence it is sufficient to show that for $\eta$-a.e.~$t\in {\Omega}$, $\underline{\dim}_H \Pi^t_*\mu\geq  \min\{d, d_\mu(t)\}$. Suppose on the contrary that this is false. Then there exist $k\in \N$ and  $A\subset {\Omega}$ with $\eta(A)>0$ such that
\begin{equation}
\label{e-f1'}
\underline{\dim}_H \Pi^t_*\mu< \min\{d, d_\mu(t)\}-\frac2k\qquad \mbox{ for all } t\in A.
\end{equation}

Take a  number $\delta\in (0, \delta_0)$ small enough such that
\begin{equation}
\label{e-f6'}\frac{ \psi(\delta)}{\log (1/\theta)}<\frac1k.
\end{equation}
Since ${\Omega}$ is a separable metric space, it has a  countable dense subset $$Y=\{y_n: \; n\in \N\}.$$
Notice that by \eqref{e-f1} and Lemma \ref{lem-var},
$$
0\leq d_\mu(t)\leq d(t)\leq \frac{ \log \ell}{ \log (1/\theta)}\quad  \mbox{ for each }t\in {\Omega}.
$$
  Due to this fact,  for each  $n\in \N$ we may pick $y_n^*\in B(y_n,\delta/2)$ so that
\begin{equation}
\label{e-f5'}
d_\mu(y_n^*)\geq \sup_{t\in B(y_n, \delta/2)}d_\mu(t)-\frac{1}{k}.
\end{equation}
Let $Y^*=\{y_{n}^*:\; n\in \N\}$. Clearly, $Y^*$ is countable.  We claim that
\begin{equation}
\label{e-f6}
\sup_{y^*\in B(t, \delta)\cap Y^*} d_\mu(y^*)\geq d_\mu(t)-\frac1k \quad \mbox{ for all } t\in {\Omega}.
\end{equation}
To see this, let $t\in {\Omega}$. Since $Y$ is dense in ${\Omega}$, there exists an integer $m$ such that $\rho(y_m, t)\leq \delta/2$.
Then
\begin{equation}\label{e-to2}
y_m^*\in B(y_m, \delta/2)\subset B(t,\delta).
\end{equation} Meanwhile by \eqref{e-f5'},
$$
d_\mu(y_m^*)\geq \sup_{t'\in B(y_m, \delta/2)}d_\mu(t')-\frac{1}{k}\geq d_\mu(t)-\frac1k,
$$
where in the last inequality we use the fact that $t\in B(y_m, \delta/2)$ (since $\rho(y_m, t)\leq \delta/2$).
This proves \eqref{e-f6}, since  $y^*_{m}\in B(t,\delta)\cap Y^*$ by \eqref{e-to2}.

Set for $n\in \N$,
\begin{equation}
\label{e-to3}
{\Omega}_n=\{t\in B(y_n^*,\delta):\; d_\mu(y^*_n)\geq d_\mu(t)-1/k\}.
\end{equation}
By \eqref{e-f6},  ${\Omega}=\bigcup_{n=1}^\infty {\Omega}_n$. Define
$$
E_n=\{t\in B(y_n^*,\delta):\; \underline{\dim}_H \Pi^t_*\mu<\min\{d, d_\mu(y_n^*)\}-1/k\},\quad n\in \N.
$$
By \eqref{e-f1'}, we see that $A\cap {\Omega}_n\subset E_n$ for each $n\in \N$.
However by Proposition \ref{pro-f1.1}(i) and \eqref{e-f6'}, for each $n\in \N$ and $\eta$-a.e.~$t\in B(y_n^*,\delta)$,
$$
\underline{\dim}_H \Pi^t_*\mu\geq \min\{d, d_\mu(y_n^*)\}-\frac{ \psi(\delta)}{\log (1/\theta)}> \min\{d, d_\mu(y_n^*)\}-\frac1k.
$$
It follows that
$\eta(E_n)=0$  for  each $n\in \N$.
Hence $$
\eta(A)=\eta\left(A\cap \left( \bigcup_{n=1}^\infty {\Omega}_n\right)\right)\leq \sum_{n=1}^\infty \eta(A\cap {\Omega}_n)\leq \sum_{n=1}^\infty \eta(E_n)=0,
$$
leading to a contradiction. This proves the statement that for $\eta$-a.e.~$t\in {\Omega}$,
$\Pi^t_*\mu$ is exact dimensional with dimension $\min\{d, d_\mu(t)\}$.

Next we prove that $\Pi^t_*\mu\ll \mathcal L_d$ for $\eta$-a.e.~$t\in \{t'\in {\Omega}:\; d_\mu(t')>d\}$.  Again we use contradiction. Suppose on the contrary that this result is false. Then there exist $k\in \N$ and $A'\subset {\Omega}$ with $\eta(A')>0$ such that
\begin{equation}
\label{e-f8'}
d_\mu(t)>d+\frac2k \quad \mbox{ and }\quad \Pi^t_*\mu\not\ll \mathcal L_d\qquad \mbox{ for all }t\in A'.
\end{equation}
  Set
$$
F_n=\{t\in B(y_n^*,\delta):\; \Pi^t_*\mu\not \ll \mathcal L_d\}, \quad n\in \N.
$$
Clearly, $A'\cap {\Omega}_n\subset F_n$ for each $n\in \N$. Since ${\Omega}=\bigcup_{n=1}^\infty {\Omega}_n$ and $\eta(A')>0$,  there exists $m\in \N$ so that $\eta(A'\cap {\Omega}_m)>0$.  Hence $A'\cap {\Omega}_m\neq \emptyset$. Pick $t\in A'\cap {\Omega}_m$. By \eqref{e-to3}, \eqref{e-f8'} and \eqref{e-f6'},
$$d_\mu(y_m^*)\geq d_\mu(t)-\frac{1}{k}>d+\frac{1}{k}>d+\frac{ \psi(\delta)}{\log (1/\theta)}.$$
 Hence $\eta(F_m)=0$ by Proposition \ref{pro-f1.1}(ii).   Since $A'\cap {\Omega}_m\subset F_m$, it follows that $\eta(A'\cap {\Omega}_m)=0$, leading to a contradiction.
\end{proof}

\begin{proof}[Proof of Theorem \ref{thm-f1.1}(ii)]
By Lemma \ref{lem-var}, for each $t\in {\Omega}$ we can find a $\sigma$-invariant ergodic measure $\mu_t$ on  $\Sigma$ such that
$$
d(t)=d_{\mu_t}(t).
$$
Moreover $0\leq d(t)\leq \log \ell/\log (1/\theta)$.

Next we prove that $\dim_H K^t=\dim_BK^t=\min\{d, d(t)\}$ for $\eta$-a.e.~$t\in {\Omega}$. By Theorem \ref{thm-FS},  $\overline{\dim}_B K^t\leq \min\{d, d(t)\}$ for every $t\in {\Omega}$. Hence it is sufficient to  show that $$\dim_H K^t\geq \min\{d, d(t)\} \quad \mbox{  for $\eta$-a.e.~$t\in {\Omega}$}.
$$
 Suppose on the contrary that this statement is false. Then there exist $k\in \N$ and $H\subset {\Omega}$ with $\eta(H)>0$ such that
\begin{equation}
\label{e-f9}
\dim_H K^t<\min\{d, d(t)\}-2/k \quad \mbox{ for all } t\in H.
\end{equation}

Take a number $\delta\in (0,\delta_0)$ such that \eqref{e-f6'} holds. Since $d(t)$ is uniformly bounded from above, similar to the construction of $Y^*$ in the proof of part (i) we can construct a countable dense subset $Y'=\{y'_n\}_{n=1}^\infty$ of ${\Omega}$ such that
\begin{equation}
\label{e-f6''}
\sup_{y'\in B(t, \delta)\cap Y'} d(y')\geq d(t)-\frac1k \quad \mbox{ for all } t\in {\Omega}.
\end{equation}

Write
\begin{equation}
\label{e-to4}{\Omega}_n'=\{t\in B(y'_n,\delta):\; d(y'_n)\geq d(t)-1/k\}\quad \mbox{ for }n\in \N.
\end{equation}
  By \eqref{e-f6''},
${\Omega}=\bigcup_{n=1}^\infty {\Omega}_n'$. Notice that for each $n\in \N$,
\begin{align*}
{\Omega}_n'\cap H&\subset \{t\in B(y_n', \delta):\; \dim_H K^t<\min\{d, d(y_n')\}-1/k\}\\
&\subset \{t\in B(y_n', \delta):\; \underline{\dim}_H \Pi^t_*(\mu_{y_n'})<\min\{d, d_{\mu_{y_n'}} (y_n')\}-1/k\}\\
&\subset \left\{t\in B(y_n', \delta):\; \underline{\dim}_H \Pi^t_*(\mu_{y_n'})<\min\{d, d_{\mu_{y_n'}} (y_n')\}-\frac{ \psi(\delta)}{\log (1/\theta)}\right\},
\end{align*}
where we have used the facts that $\dim_H K^t\geq \underline{\dim}_H \Pi^t_*(\mu_{y_n'})$ and
$d(y_n')=d_{\mu_{y_n'}} (y_n')$ in the second inclusion, and \eqref{e-f6'} in the last inclusion. Hence
$\eta({\Omega}_n'\cap H)=0$ for each $n$ by applying Proposition \ref{pro-f1.1}(i).  It follows that $\eta(H)\leq \sum_{n=1}^\infty \eta({\Omega}_n'\cap H)=0$, leading to a contradiction. This completes the proof of the statement that $\dim_H K^t=\min\{d, d(t)\}$ for $\eta$-a.e.~$t\in {\Omega}$.

Finally we prove that $\mathcal L_d(K^t)>0$ for $\eta$-a.e.~$t\in \{t'\in {\Omega}:\; d(t')>d\}$.  Suppose on the contrary that this result is false. Then there exist $k\in \N$ and $H'\subset {\Omega}$ with $\eta(H')>0$ such that
\begin{equation}
\label{e-f8}
d(t)>d+\frac2k \quad \mbox{ and }\quad \mathcal L_d(K^t)=0\qquad \mbox{ for all }t\in H'.
\end{equation}
  Set
$$
F_n'=\{t\in B(y_n',\delta):\; \mathcal L_d(K^t)=0\}, \quad n\in \N.
$$
Clearly, $H'\cap {\Omega}_n'\subset F'_n$ for each $n\in \N$. Since ${\Omega}=\bigcup_{n=1}^\infty {\Omega}_n'$ and $\eta(H')>0$,  there exists $m\in \N$ so that $\eta(H'\cap {\Omega}_m')>0$.  Hence $H'\cap {\Omega}_m'\neq \emptyset$. Taking $t\in H'\cap {\Omega}_m'$ and applying   \eqref{e-to4}, \eqref{e-f8} and \eqref{e-f6'} gives
$$d_{\mu_{y_m'}}(y_m')=d(y_m')\geq d(t)-\frac1k>d+\frac{1}{k}>d+\frac{ \psi(\delta)}{\log (1/\theta)}.$$
 Now by Proposition \ref{pro-f1.1}(ii), $\Pi^t_*(\mu_{y_m'})\ll \mathcal L_d$ for $\eta$-a.e.~$t\in B(y_m', \delta)$.  This implies that  $\eta(F_m')=0$.  Since $H'\cap {\Omega}_m'\subset F_m'$, it follows that $\eta(H'\cap {\Omega}_m')=0$, leading to a contradiction.
\end{proof}

\section{Translational family of  IFSs generated by a dominated lower triangular $C^1$ IFS}
\label{S-triangular}
In this section, we show that under mild assumptions, a translational family of $C^1$ IFSs generated by a dominated lower triangular $C^1$ IFS, satisfies the GTC. To begin with, let $S$ be a compact  subset of  $\R^d$ with non-empty interior.

\begin{de}
{\rm
Let $\ell\in \N$ with $\ell\geq 2$. We say that $\mathcal F=\{f_i\}_{i=1}^\ell$ a {\it dominated lower triangular $C^1$ IFS} on $S$ if the following conditions hold:
\begin{itemize}
\item[(i)] $f_i(S)\subset {\rm int} (S)$, $i=1,\ldots, \ell$.
\item[(ii)] There exists a bounded open connected set $U\supset S$ such that each  $f_i$ extends to a contracting $C^1$ diffeomorphism $f_i: U\to f_i(U)$ with $\overline{f_i(U)}\subset U$.
\item[(iii)] For each $z\in S$ and $i\in \{1,\ldots, \ell\}$, the Jacobian matrix $D_zf_i$ of $f_i$ at $z$ is a lower triangular matrix such that
    $$
    |(D_zf_i)_{jj}|\leq |(D_zf_i)_{kk}|\qquad  \mbox{ for all }1\leq k\leq j\leq d.
    $$
\end{itemize}
}
\end{de}

In the remaining part of this section, we fix a dominated lower triangular $C^1$ IFS $\mathcal F=\{f_i\}_{i=1}^\ell$ on $S$.

By continuity, there exists a small $r_0>0$ such that the following holds. Setting $$f_i^{\pmb{\mathfrak{t}}}:=f_i+\mathbf{t}_i$$ for $1\leq i\leq \ell$ and  $\pmb{\mathfrak{t}}=(\mathbf{t}_1,\ldots, \mathbf{t}_\ell)\in \R^{\ell d}$ with $|\pmb{\mathfrak{t}}|< r_0$, we have $f_i^{\pmb{\mathfrak{t}}}(S)\subset {\rm int}(S)$ for each $i$.

Write $\Delta:=\{\pmb{\mathfrak{t}}\in \R^{\ell d}:\; |\pmb{\mathfrak{t}}|<r_0\}$ and set
$$
\mathcal F^{\pmb{\mathfrak{t}}}=\{f_i^{\pmb{\mathfrak{t}}}\}_{i=1}^\ell,\qquad \pmb{\mathfrak{t}}\in \Delta.
$$
We call  $\mathcal F^{\pmb{\mathfrak{t}}}$, ${\pmb{\mathfrak{t}}\in \Delta}$, a {\it translational family of IFSs  generated by $\mathcal F$}.  For $\bi=i_1\ldots i_n\in \Sigma_n$, we write $f_{\bi}^{\pmb{\mathfrak{t}}}=f_{i_1}^{\pmb{\mathfrak{t}}}\circ\cdots\circ
f_{i_n}^{\pmb{\mathfrak{t}}}$.

For a $C^1$ map $g: S\to \R^d$ and $z_1,\ldots, z_d\in S$, we write
\begin{equation}
\label{e-to5}
D_{z_1,\ldots, z_d}^*g=\left[\begin{array}{c}
\nabla^Tg_1(z_1) \\
\vdots \\
\nabla^Tg_d(z_d)
\end{array}\right]=\left[
\begin{array}{ccc}
\frac{\partial g_1}{\partial x_1}(z_1) & \cdots &\frac{\partial g_1}{\partial x_d}(z_1)\\
\vdots & \ddots & \vdots \\
\frac{\partial g_d}{\partial x_1}(z_d) & \cdots &\frac{\partial g_d}{\partial x_d}(z_d)
\end{array}
\right],
\end{equation}
where $g_i$ is the $i$-th component of the map $g$, $i=1,\ldots, d$. Clearly,
\begin{equation}
\label{e-to6}
(D_{z_1,\ldots, z_d}^*g)_{ij}=(D_{z_i}g)_{ij}\quad \mbox{ for all }1\leq i,j\leq d.
\end{equation}

The main result in this section is the following.

\begin{thm}
\label{thm-3.3} Let $\mathcal F^{\pmb{\mathfrak{t}}}$, ${\pmb{\mathfrak{t}}\in \Delta}$, be a  translational family of IFSs  generated by a dominated lower triangular $C^1$ IFS $\mathcal F$ defined on a compact convex subset $S$ of $\R^d$.  Suppose in addition that
\begin{equation}
\label{e-con}
\rho:=\max_{1\leq i,j\leq \ell:\; i\neq j} \left(\sup_{y\in S} \|D_yf_i\|+\sup_{z\in S} \|D_zf_j\|\right)<1.
\end{equation}
  Then
$\mathcal F^{\pmb{\mathfrak{t}}}$, $\pmb{\mathfrak{t}}\in \Delta$, satisfies the GTC with respect to  $\ell d$-dimensional Lebesgue measure $\mathcal L_{\ell d}$ restricted {\color{red} to} $\Delta$.
\end{thm}

The proof of the above theorem is based on the following.
\begin{pro}
\label{pro-3.2}Let $\mathcal F^{\pmb{\mathfrak{t}}}$, ${\pmb{\mathfrak{t}}\in \Delta}$, be a  translational family of IFSs  generated by a dominated lower triangular $C^1$ IFS $\mathcal F$  on a compact  subset $S$ of $\R^d$. Then there exists a function $h: (0, r_0)\to (0,\infty)$ with $\lim_{\delta\to 0} h(\delta)=0$ such that for each $\delta\in (0,r_0)$, there is  $C(\delta)\geq 1$ so that
\begin{equation}
\label{e-e3.2}
\| D_y f_{\pmb{\omega}}^{\pmb{\mathfrak{s}}}\cdot(D^*_{z_1,\ldots, z_d}  f_{\pmb{\omega}}^{\pmb{\mathfrak{t}}})^{-1}\|\leq C(\delta) e^{nh(\delta)}
\end{equation}
for every $n\in \N$, $\pmb{\omega}\in \Sigma_n$, $y, z_1,\ldots, z_d\in S$ and $\pmb{\mathfrak{s}},\pmb{\mathfrak{t}}\in \Delta$ with
$|\pmb{\mathfrak{s}}-\pmb{\mathfrak{t}}|<\delta$.
\end{pro}

In the next two subsections we prove Proposition \ref{pro-3.2} and Theorem \ref{thm-3.3} respectively.

\subsection{Proof of Proposition \ref{pro-3.2}}

We first prove several auxiliary lemmas.
 \begin{lem}
\label{lem-5.1}
Let $c\geq 1$ and $d\in \N$. Let $A$ be a real $d\times d$ non-singular lower triangular matrix such that
\begin{equation}
\label{e-f4.1}
|A_{ij}|\leq c|A_{jj}|\quad \mbox{ for all }1\leq i,j\leq d.
\end{equation}
Then
$$
|(A^{-1})_{ij}|\leq (c\sqrt{d})^{d-1}|(A^{-1})_{ii}|\quad  \mbox{ for all }1\leq i,j\leq d.
$$
\end{lem}
\begin{proof}
It is well known (see e.g.~\cite{HornJohnson1985}) that $A^{-1}=\frac{1}{\det(A)} \mbox{adj}(A)$, where $\mbox{adj}(A)$ is the adjugate matrix of $A$ defined by
$$
(\mbox{adj}(A))_{ij}=(-1)^{i+j}\det(A(j,i)),\qquad 1\leq i,j\leq d,
$$
here $A(j,i)$ is the $(d-1)\times (d-1)$ matrix that results from $A$ by removing the $j$-th row and $i$-th column. By the Hadamard's inequality (see e.g.~\cite[Corollary 7.8.2]{HornJohnson1985}), $|\det(A(j,i))|$ is bounded above by the product of the Euclidean norms of the columns of $A(j,i)$.  In particular, this implies that
$$
|\det(A(j,i))|\leq \prod_{1\leq k\leq d:\; k\neq i}|v_k|,
$$
where $v_k$ denotes the $k$-th column vector of $A$.  By \eqref{e-f4.1},
$$|v_k|=\sqrt{\sum_{i=1}^d(A_{ik})^2}\leq c\sqrt{d}|A_{kk}|,$$
so $|\det(A(j,i))|\leq (c\sqrt{d})^{d-1} \prod_{1\leq k\leq d:\; k\neq i} |A_{kk}|$.
Hence for given $1\leq i,j\leq d$,
$$
\frac{|(A^{-1})_{ij}|}{|(A^{-1})_{ii}|}=\frac{|\det(A(j,i)|}{\det(A)\cdot |(A^{-1})_{ii}|}\leq \frac{(c\sqrt{d})^{d-1} \prod_{1\leq k\leq d:\; k\neq i} |A_{kk}|}{\det(A)\cdot |(A^{-1})_{ii}|}=(c\sqrt{d})^{d-1}.
$$
\end{proof}

For $c\geq 1$ and $d\in \N$, let $\mathcal T_c(d)$ denote the collection of real $d\times d$  lower triangular matrices $A=(a_{ij})$ satisfying the following two conditions: \begin{itemize}
\item[(i)] $|a_{11}|\geq |a_{22}|\geq \ldots \geq |a_{dd}|>0$;
\item[(ii)] $|a_{ij}|\leq c|a_{jj}|$ for all $ 1\leq i,j\leq d$.
\end{itemize}
Then we have the following estimates.
\begin{lem}
\label{lem-a2}
Let $n\in \N$ and $A_1,\ldots, A_n\in \mathcal T_c(d)$. Then for  $1\leq j\leq i\leq d$,
\begin{equation}\label{e-f4.2}
|(A_1\cdots A_n)_{ij}|\leq (cn)^{i-j} |(A_1\cdots A_n)_{jj}|.
\end{equation}
\end{lem}
\begin{proof}
We prove by induction on $n$. Since $A_1\in \mathcal T_c(d)$, the inequality \eqref{e-f4.2} holds when $n=1$. Now assume that \eqref{e-f4.2} holds  when $n=k$. Below we show that it also holds when $n=k+1$.

Given $A_1,\ldots, A_{k+1}\in \mathcal T_c(d)$, we write $A=A_1$ and $B=A_2\cdots A_{k+1}$. Clearly $B$ is lower triangular. By the induction assumption, $|B_{ij}|\leq (ck)^{i-j} |B_{jj}|$ for each pair $(i,j)$ with $1\leq j\leq i\leq d$.

Now fix a pair $(i,j)$ with $1\leq j\leq i\leq d$.
Observe that
\begin{equation}
\label{e-f4.3}
\frac{(AB)_{ij}}{(AB)_{jj}}=\sum_{ p=j}^i\frac{A_{ip}}{A_{jj}} \cdot \frac{B_{pj}}{B_{jj}}=\frac{A_{ii}}{A_{jj}} \cdot \frac{B_{ij}}{B_{jj}}+\sum_{p=j}^{i-1} \frac{A_{ip}}{A_{jj}} \cdot \frac{B_{pj}}{B_{jj}}.
\end{equation}
Applying the inequalities $|A_{ii}|\leq |A_{jj}|$, $|B_{ij}|\leq (ck)^{i-j}|B_{jj}|$, $|A_{ip}|\leq c|A_{pp}|\leq c|A_{jj}|$ and
$|B_{pj}|\leq (ck)^{p-j}|B_{jj}|$ to \eqref{e-f4.3} gives
 $$
 \left|\frac{(AB)_{ij}}{(AB)_{jj}}\right|\leq (ck)^{i-j} + c\sum_{p=j}^{i-1}(ck)^{p-j}\leq (c(k+1))^{i-j}.
 $$
 Hence \eqref{e-f4.2} holds for $n=k+1$.
\end{proof}

\begin{lem}\label{lem-b20}
 Let $\mathcal F^{\pmb{\mathfrak{t}}}=\{f_i^{\tt}\}_{i=1}^\ell$, ${\pmb{\mathfrak{t}}\in \Delta}$, be a  translational family of IFSs on a compact subset $S$ of $\R^d$ generated by a  $C^1$ IFS $\mathcal F=\{f_i\}_{i=1}^\ell$.  Let $\theta\in (0,1)$ be a common Lipschitz constant of $f_1,\ldots, f_\ell$ on $S$. That is,
  $$
  |f_i(u)-f_i(v)|\leq \theta|u-v| \quad \mbox{ for all }1\leq i\leq \ell\mbox{ and }u, v\in S.
  $$
Then for  ${\mathfrak{s}}, {\mathfrak{t}}\in \Delta$, $u,v\in S$,  $n\in \N$ and $\pmb{\tau}\in\Sigma_{n}$,
\begin{equation}\label{b17}
|
  f^{\pmb{\mathfrak{t}} }_{\pmb{\tau}}(u)
  -
  f^{\pmb{\mathfrak{s}} }_{\pmb{\tau}}(v)
  | \leq
    \frac{|\pmb{\mathfrak{t}} -\pmb{\mathfrak{s}}|}{1-\theta}
    +
    \theta^{n}
    \left(
    |u-v|-
    \frac{|\pmb{\mathfrak{t}} -\pmb{\mathfrak{s}}|}{1-\theta}
    \right).
  \end{equation}
  In particular,
  \begin{equation}\label{b62}
  |
  f^{\pmb{\mathfrak{t}} }_{\pmb{\tau}}(u)
  -
  f^{\pmb{\mathfrak{s}}}_{\pmb{\tau}}(u)
  | \leq
    \frac{|\pmb{\mathfrak{t}} -\pmb{\mathfrak{s}}|}{1-\theta}\quad \mbox{ and } \quad
    |
  f^{\pmb{\mathfrak{t}} }_{\pmb{\tau}}(u)
  -
  f^{\pmb{\mathfrak{t}}}_{\pmb{\tau}}(v)
  | \leq
    \theta^n |u-v|.
  \end{equation}

\end{lem}

\begin{proof}
  To verify \eqref{b17} we let $i\in \{1,\ldots, \ell\}$. Then
  	\begin{eqnarray*}
|
f^{\pmb{\mathfrak{t}} }_i(u)-f^{\pmb{\mathfrak{s}}}_i(v)
|
&\leq &
|
f^{\pmb{\mathfrak{t}} }_i(u)-f^{\pmb{\mathfrak{s}} }_i(u)
|
+
|
f^{\pmb{\mathfrak{s}} }_i(u)-f^{\pmb{\mathfrak{s}} }_i(v)
|
\\
&= &
|
\mathbf{t}_i-\mathbf{s}_i
|
+
|
f_i(u)-f_i(v)
|
\\
&\leq &
|
\pmb{\mathfrak{t}}-\pmb{\mathfrak{s}}
|
+\theta|u-v|.
\end{eqnarray*}
Let $\varphi:\;\R\to \R$ be a contracting affine  map defined by $\varphi(x)= |
\pmb{\mathfrak{t}}-\pmb{\mathfrak{s}}
| +\theta x$ {\color{red} for given $\ss$ and $\tt$}.  Then for every $1\leq i\leq \ell$,
\begin{equation}\label{b61}
	|
f^{\pmb{\mathfrak{t}} }_i(u)-f^{\pmb{\mathfrak{s}}}_i(v)
|
\leq
\varphi(|u-v|).
\end{equation}

Now we can prove \eqref{b17} by using the above inequality.
Indeed, using \eqref{b61} and the fact that $\varphi(\cdot)$ is monotone increasing, we obtain
that for  $1\leq i,j\leq \ell$,
\[
|
f^{\pmb{\mathfrak{t}} }_j(
f^{\pmb{\mathfrak{t}} }_i(u)
)
-
f^{\pmb{\mathfrak{s}} }_j(
f^{\pmb{\mathfrak{s}} }_i(v)
)
|
\leq
\varphi\left(
  |
  f^{\pmb{\mathfrak{t}} }_{i}(u)
  -
  f^{\pmb{\mathfrak{s}}}_{i}(v) |
  \right)
\leq
\varphi^2(|u-v|).
\]
Successive application of this implies that for every $\pmb{\tau}\in\Sigma_{n}$ and $u,v\in S$,
$$
|
  f^{\pmb{\mathfrak{t}} }_{\pmb{\tau}}(u)
  -
  f^{\pmb{\mathfrak{s}} }_{\pmb{\tau}}(v)
  | \leq
   \varphi^{n}(|u-v|)=\frac{|\pmb{\mathfrak{t}} -\pmb{\mathfrak{s}}|}{1-\theta}
    +
    \theta^{n}
    \left(
    |u-v|-
    \frac{|\pmb{\mathfrak{t}} -\pmb{\mathfrak{s}}|}{1-\theta}
    \right).
$$
This proves \eqref{b17}.
The assertions in \eqref{b62} then follow directly from  \eqref{b17}.
\end{proof}

Now we are ready to prove Proposition \ref{pro-3.2}.
\begin{proof}[Proof of Proposition \ref{pro-3.2}] We divide the proof into 5 small steps.

{\sl Step 1}. Write
\begin{equation}
\label{cn0}
C_n:=\sup\left\{ \left|\frac{(D_yf^{\tt}_{{\pmb{\omega}}} )_{ii}}{(D_zf^{\tt}_{{\pmb{\omega}}} )_{ii}}\right|:\; \tt\in \Delta,\; y,z\in S, \;{\pmb{\omega}}\in \Sigma_n,\; 1\leq i\leq d\right\}.
\end{equation}
We claim that
\begin{equation}
\label{e-cn}
\lim_{n\to \infty}\frac{1}{n}\log C_n=0.
\end{equation}

To prove this claim, for each $p\in \{1,\ldots, \ell\}$ and $i\in \{1,\ldots, d\}$ we define a function $a_{p,i}:\; S\to \R$ by
$$
a_{p,i}(z)=\log \left|(D_zf_p)_{ii}\right|.
$$
Clearly, the functions $a_{p,i}$ are continuous on $S$. Since the matrix $D_zf_p$ is lower triangular for each $z\in S$ and $1\leq p\leq \ell$, it follows that for $\tt\in \Delta$, $y,z\in S$, ${\pmb{\omega}}=\omega_1\cdots\omega_n\in \Sigma_n$ and $1\leq i\leq d$,
\begin{equation}
\label{e-co1}
\log |(D_zf^{\tt}_{{\pmb{\omega}}} )_{ii}|=\sum_{k=1}^n a_{\omega_k,i}(f^{\tt}_{\sigma^k{{\pmb{\omega}}}}(z))
\end{equation}
and
\begin{equation}
\label{e-co}
\log \left|\frac{(D_yf^{\tt}_{{\pmb{\omega}}} )_{ii}}{(D_zf^{\tt}_{{\pmb{\omega}}} )_{ii}}\right|=\sum_{k=1}^n \left(a_{\omega_k,i}(f^{\tt}_{\sigma^k{{\pmb{\omega}}}}(y))-
a_{\omega_k,i}(f^{\tt}_{\sigma^k{{\pmb{\omega}}}}(z))\right),
\end{equation}
where $\sigma^k{{\pmb{\omega}}}:=\omega_{k+1}\cdots \omega_n$ for $1\leq k\leq n-1$,  $\sigma^n{{\pmb{\omega}}}:=\varepsilon$ (here $\varepsilon$ stands for the empty word) and  $f^{\tt}_{\varepsilon}(y):=y$. Define
$\gamma: (0,\infty)\to (0,\infty)$ by
\begin{equation}
\label{e-to7}
\gamma(u)=\max_{1\leq p\leq \ell,\;1\leq i\leq d}\sup \{|a_{p,i}(y)-a_{p,i}(z)|:\; y,z\in S,\; |y-z|\leq u\}.
\end{equation}
Since $S$ is compact and $a_{p,i}$ are continuous, it follows that $\lim_{u\to 0}\gamma(u)=0$.  To estimate the term in the lefthand side of the equality \eqref{e-co}, by Lemma \ref{lem-b20} we obtain that
$$|f^{\tt}_{\sigma^k{{\pmb{\omega}}}}(y)-f^{\tt}_{\sigma^k{{\pmb{\omega}}}}(z)|\leq \theta^{n-k}|y-z|\leq \theta^{n-k}{\rm diam}(S),
$$
 where $\theta\in (0,1)$ is a common Lipschitz constant of $f_1,\ldots, f_\ell$ on $S$.  Hence by \eqref{e-co},
 $$
\frac{1}{n} \log \left|\frac{(D_yf^{\tt}_{{\pmb{\omega}}} )_{ii}}{(D_zf^{\tt}_{{\pmb{\omega}}} )_{ii}}\right|\leq \frac{1}{n} \sum_{k=1}^n\gamma\left(\theta^{n-k}{\rm diam}(S)\right)\to 0, \mbox{ as }n\to \infty.
 $$
 This proves \eqref{e-cn}.

 {\sl Step 2}. For  $\ss,\tt\in \Delta$, $y\in S$, $n\in \N$, ${\pmb{\omega}}\in \Sigma_n$ and $1\leq i\leq d$,
 \begin{equation}
 \label{e-cn1}
 \left|\frac{(D_yf^{\tt}_{{\pmb{\omega}}} )_{ii}}{(D_yf^{\ss}_{{\pmb{\omega}}} )_{ii}}\right|\leq \exp\left(n\gamma\left(\frac{|\tt-\ss|}{1-\theta}\right)\right),
 \end{equation}
 where $\gamma(\cdot)$ is defined as in \eqref{e-to7} and $\theta\in (0,1)$ is a common Lipschitz constant for $f_1,\ldots, f_\ell$ on $S$.

 To prove \eqref{e-cn1}, by \eqref{e-co1} we see that
 \begin{eqnarray*}
 \log \left|\frac{(D_yf^{\tt}_{{\pmb{\omega}}} )_{ii}}{(D_yf^{\ss}_{{\pmb{\omega}}} )_{ii}}\right|&=&\sum_{k=1}^n \left(a_{\omega_k,i}(f^{\tt}_{\sigma^k{{\pmb{\omega}}}}(y))-
a_{\omega_k,i}(f^{\ss}_{\sigma^k{{\pmb{\omega}}}}(y))\right)\\
&\leq & n \gamma \left(\frac{|\tt-\ss|}{1-\theta}\right),
 \end{eqnarray*}
 where in the second inequality we have used the fact that $|f^{\tt}_{\sigma^k{{\pmb{\omega}}}}(y)-f^{\ss}_{\sigma^k{{\pmb{\omega}}}}(y)|\leq \frac{|\tt-\ss|}{1-\theta}$ (which follows from  \eqref{b62}).  This proves \eqref{e-cn1}.

 {\sl Step 3}. Set
 \begin{equation}
 \label{e-cn4}
 c=\sup\left\{\left|\frac{(D_yf_p)_{ij}}{(D_yf_p)_{jj}}\right|:\; y\in S,\; 1\leq p\leq \ell,\; 1\leq i,j\leq d\right\}.
 \end{equation}
 Then
 \begin{equation}
 \label{e-cn5}
 \left|\frac{(D_yf_{\pmb{\omega}}^{\ss})_{ij}}{(D_yf_{\pmb{\omega}}^{\ss})_{jj}} \right|\leq (cn)^d \quad \mbox{ for all }
 \ss\in \Delta,\; y\in S,\; \pmb{\omega}\in \Sigma_n,\; 1\leq i,j\leq d.
 \end{equation}
 To see this, we simply notice that $D_yf_{{\pmb{\omega}}}^{\ss}=\prod_{k=1}^n D_{f^{\ss}_{\sigma^k {\pmb{\omega}}}(y)}f_{\omega_k}$ and apply Lemma \ref{lem-a2}.

 {\sl Step 4}. Let $c$ and $C_n$ be defined as in \eqref{e-cn4} and \eqref{cn0}. Then for  $\tt\in \Delta$, $y, z_1,\ldots, z_d\in S$, $\pmb{\omega}\in \Sigma_n$ and $1\leq k,j\leq d$,
 \begin{equation}
 \label{e-cn6}
 \left|\left(\left(D_{z_1,\ldots, z_d}^*f^{\tt}_{{\pmb{\omega}}}\right)^{-1}\right)_{kj}\right|
 \leq (cn)^{d(d-1)}(\sqrt{d})^{d-1}(C_n)^d
 \frac{1}{\left|\left(D_yf^{\tt}_{{\pmb{\omega}}}\right)_{kk}\right|},
 \end{equation}
 where $D_{z_1,\ldots, z_d}^*g$ is defined as in \eqref{e-to5}.

 To prove \eqref{e-cn6}, notice that for all $1\leq k,j\leq d$,
 \begin{eqnarray*}
 \left|\left(D_{z_1,\ldots, z_d}^*f^{\tt}_{{\pmb{\omega}}}\right)_{kj}\right|&=& \left|\left(D_{z_k}f^{\tt}_{{\pmb{\omega}}}\right)_{kj}\right| \qquad \mbox{(by \eqref{e-to6})}\\
 &\leq& (cn)^d  \left|\left(D_{z_k}f^{\tt}_{{\pmb{\omega}}}\right)_{jj}\right| \qquad \mbox{(by \eqref{e-cn5})}\\
 &\leq& (cn)^d C_n \left|\left(D_{z_j}f^{\tt}_{{\pmb{\omega}}}\right)_{jj}\right| \qquad \mbox{(by \eqref{cn0})}\\
 &=& (cn)^d C_n \left|\left(D_{z_1,\ldots, z_d}^*f^{\tt}_{{\pmb{\omega}}}\right)_{jj}\right|\qquad \mbox{(by \eqref{e-to6})}.
 \end{eqnarray*}
 Applying Lemma \ref{lem-5.1} (in which we replace $c$ by $(cn)^dC_n$ and take $A=D_{z_1,\ldots, z_d}^*f^{\tt}_{{\pmb{\omega}}}$), we obtain
 \begin{equation*}
 \begin{split}
\left|\left(\left(D_{z_1,\ldots, z_d}^*f^{\tt}_{{\pmb{\omega}}}\right)^{-1}\right)_{kj}\right|&\leq \left(cn)^dC_n\sqrt{d}\right)^{d-1}\left|\left(\left(D_{z_1,\ldots, z_d}^*f^{\tt}_{{\pmb{\omega}}}\right)^{-1}\right)_{kk}\right|\\
& =\left((cn)^dC_n\sqrt{d}\right)^{d-1}\left|\left(\left(D_{z_k}f^{\tt}_{{\pmb{\omega}}}\right)^{-1}\right)_{kk}\right|\\
& =\left((cn)^dC_n\sqrt{d}\right)^{d-1}\frac{1}{\left|\left(D_{z_k}f^{\tt}_{{\pmb{\omega}}}\right)_{kk}\right|}\\
&\leq \left((cn)^dC_n\sqrt{d}\right)^{d-1}C_n\frac{1}{\left|\left(D_{y}f^{\tt}_{{\pmb{\omega}}}\right)_{kk}\right|},
 \end{split}
 \end{equation*}
 from which \eqref{e-cn6} follows.

 {\sl Step 5}. Now we are ready to prove \eqref{e-e3.2}. Let $\delta\in (0,r_0)$. Write $$u_n:=(cn)^{d(d-1)}(\sqrt{d})^{d-1}(C_n)^d,\qquad n\in \N.$$
   Then for $\ss,\tt\in \Delta$ with $|\tt-\ss|\leq \delta$ and
 $1\leq i,j\leq d$,
 \begin{eqnarray*}
\left| \left(D_y f_{\pmb{\omega}}^{\pmb{\mathfrak{s}}}\cdot(D^*_{z_1,\ldots, z_d}  f_{\pmb{\omega}}^{\pmb{\mathfrak{t}}})^{-1}\right)_{ij}\right|
&\leq & \sum_{k=1}^d\left| \left(D_y f_{\pmb{\omega}}^{\pmb{\mathfrak{s}}}\right)_{ik}\right|\cdot \left| \left(D^*_{z_1,\ldots, z_d}  f_{\pmb{\omega}}^{\pmb{\mathfrak{t}}})^{-1}\right)_{kj} \right|\\
&\leq & (cn)^du_n\sum_{k=1}^d\left| \frac{\left(D_y f_{\pmb{\omega}}^{\pmb{\mathfrak{s}}}\right)_{kk}}{\left(D_y f_{\pmb{\omega}}^{\pmb{\mathfrak{t}}}\right)_{kk}}\right|\quad \mbox{ (by \eqref{e-cn5}, \eqref{e-cn6})}\\
&\leq &d(cn)^du_n\exp\left(n \gamma \left(\frac{|\tt-\ss|}{1-\theta}\right)\right)\quad \mbox{ (by \eqref{e-cn1})}\\
&\leq &d(cn)^du_n\exp\left(n \gamma \left(\frac{\delta}{1-\theta}\right)\right).
 \end{eqnarray*}
 {\color{red} This} implies that
\begin{equation}
\label{e-cn7}
 \left\| D_y f_{\pmb{\omega}}^{\pmb{\mathfrak{s}}}\cdot(D^*_{z_1,\ldots, z_d}  f_{\pmb{\omega}}^{\pmb{\mathfrak{t}}})^{-1}\right\|\leq d^2(cn)^du_n\exp\left(n \gamma \left(\frac{\delta}{1-\theta}\right)\right),
\end{equation}
where we have used an easily checked fact that $$\|A\|\leq d\max_{1\leq i,j\leq d}|A_{ij}|$$ for $A=(A_{ij})\in \R^{d\times d}$.

 Set $h:(0,r_0)\to (0,\infty)$ by $h(x)=x+ \gamma \left(\frac{x}{1-\theta}\right)$. Since
 $$
 \lim_{n\to \infty}\frac{1}{n}\log \left(d^2(cn)^du_n\right)=0,
 $$
 there exists  $C(\delta)>0$ such that $d^2(cn)^du_ne^{-n\delta}\leq C(\delta)$ for all $n\geq 1$. According to this fact and \eqref{e-cn7},
 we obtain  the desired inequality  $$ \left\| D_y f_{\pmb{\omega}}^{\pmb{\mathfrak{s}}}\cdot(D^*_{z_1,\ldots, z_d}  f_{\pmb{\omega}}^{\pmb{\mathfrak{t}}})^{-1}\right\|\leq C(\delta) \exp(nh(\delta));$$  for later convenience we may assume that $C(\delta)\geq 1$.
\end{proof}

\subsection{Proof of Theorem \ref{thm-3.3}}

The following result plays a key part in our proof.
\begin{lem}
\label{lem-key}
Let $\{\mathcal F^{\tt}\}_{\tt\in \Delta}$ be a translational family of IFSs on a compact set $S\subset \R^d$ generated by a $C^1$ IFS $\mathcal F=\{f_i\}_{i=1}^\ell$. Suppose that \eqref{e-con} holds.
Let $\delta>0$. Then there exists $\widetilde{C}>0$ which depends on $\mathcal F$ and $\delta$ such that the following holds.
Let $\ba=(a_n)_{n=1}^\infty,\bb=(b_n)_{n=1}^\infty\in \Sigma$ with $a_1\neq b_1$, and let $A$ be a real invertible $d\times d$ matrix.  Then for  $\ss\in \Delta$ and $r>0$,
\begin{equation}
\label{e-}
\begin{split}
\mathcal L_{\ell d}&\left\{\tt\in B_{\R^{\ell d}}(\ss, \delta)\cap \Delta:\; \Pi^{\tt}(\ba)-\Pi^{\tt}(\bb)\in A^{-1}B_{\R^d}(0,r)\right\} \\
&\leq \widetilde{C}\min\left\{\frac{r^k}{\phi^k(A)}:\; k=0,1,\ldots,d\right\},
\end{split}
\end{equation}
where $B_{\R^{\ell d}}(\cdot,\cdot)$ and $B_{\R^d}(\cdot,\cdot)$ stand for closed balls in $\R^{\ell d}$ and $\R^{d}$, repectively.
\end{lem}
Since the proof of the above lemma is a  little long, we will postpone it until we have finished the proof of  Theorem \ref{thm-3.3}.

\begin{proof}[Proof of Theorem \ref{thm-3.3} by assuming Lemma \ref{lem-key}] Fix $\ss\in \Delta$ and $\delta\in (0,r_0)$. Let $\bi,\bj\in \Sigma$ with $\bi\neq \bj$. Set $$\pmb{\omega}=\bi\wedge \bj\quad \mbox{ and }\quad n=|\pmb{\omega}|.$$
Write $\ba=\sigma^n \bi$ and $\bb=\sigma^n \bj$. Clearly $a_1\neq b_1$.

Fix $y\in S$. We claim that for $r>0$,
\begin{equation}
\label{e-split}
\begin{split}
&\left\{\tt \in B_{\R^{\ell d}}(\ss,\delta)\cap \Delta:\; |\Pi^{\tt}(\bi)-\Pi^{\tt}(\bj)|< r\right\}\\
&\mbox{}\; \subset
\left\{\tt\in B_{\R^{\ell d}}(\ss,\delta)\cap \Delta:\; \Pi^{\tt}(\ba)-\Pi^{\tt}(\bb)\in \left(D_yf^{\ss}_{\pmb{\omega}}\right)^{-1}
B_{\R^d}\left(0,C(\delta)e^{nh(\delta)}r\right) \right\},
\end{split}
\end{equation}
where $C(\delta)$ and $h(\delta)$ are given as in Proposition \ref{pro-3.2}.

To show \eqref{e-split}, let $\tt \in B_{\R^{\ell d}}(\ss,\delta)\cap \Delta$ so that $|\Pi^{\tt}(\bi)-\Pi^{\tt}(\bj)|<r$. Notice that
$$\Pi^{\tt}(\bi)-\Pi^{\tt}(\bj)=f^{\tt}_{\pmb{\omega}}(\Pi^{\tt}(\ba))-f^{\tt}_{\pmb{\omega}}(\Pi^{\tt}(\bb)).$$
Since $S$ is convex, by the mean value theorem there exist $z_1,\ldots, z_d\in S$ such that
$$\Pi^{\tt}(\bi)-\Pi^{\tt}(\bj)=\left(D_{z_1,\ldots, z_d}^*f^{\tt}_{\pmb{\omega}}\right)(\Pi^{\tt}(\ba)-\Pi^{\tt}(\bb)).
$$
Hence
\begin{eqnarray*}
\Pi^{\tt}(\ba)-\Pi^{\tt}(\bb)&=&\left(D_{z_1,\ldots, z_d}^*f^{\tt}_{\pmb{\omega}}\right)^{-1}(\Pi^{\tt}(\bi)-\Pi^{\tt}(\bj))\\
&\in& \left(D_{z_1,\ldots, z_d}^*f^{\tt}_{\pmb{\omega}}\right)^{-1}B_{\R^d}(0,r)\\
&=& \left(D_yf^{\ss}_{\pmb{\omega}}\right)^{-1} D_yf^{\ss}_{\pmb{\omega}} \left(D_{z_1,\ldots, z_d}^*f^{\tt}_{\pmb{\omega}}\right)^{-1}B_{\R^d}(0,r)\\
&\subset & \left(D_yf^{\ss}_{\pmb{\omega}}\right)^{-1} B_{\R^d}\left(0,C(\delta)e^{nh(\delta)}r\right)\quad  \mbox{ (by Proposition \ref{pro-3.2}).}
\end{eqnarray*}
This proves \eqref{e-split}.

 By \eqref{e-split} and Lemma \ref{lem-key}, we see that
\begin{equation*}
\begin{split}
\mathcal L_{\ell d}&\left\{\tt \in B_{\R^{\ell d}}(\ss,\delta)\cap \Delta:\; |\Pi^{\tt}(\bi)-\Pi^{\tt}(\bj)|< r\right\}\\
&\leq \mathcal L_{\ell d} \left\{\tt\in B_{\R^{\ell d}}(\ss,\delta)\cap \Delta:\; \Pi^{\tt}(\ba)-\Pi^{\tt}(\bb)\in \left(D_yf^{\ss}_{\pmb{\omega}}\right)^{-1}
B_{\R^d}\left(0,C(\delta)e^{nh(\delta)}r\right) \right\}\\
&\leq \widetilde{C}\cdot\min\left\{\frac{C(\delta)^ke^{nkh(\delta)}r^k}{\phi^k(D_yf^{\ss}_{\pmb{\omega}})}:\; k=0,1,\ldots,d\right\}\\
&\leq \widetilde{C}C(\delta)^de^{ndh(\delta)}\min\left\{\frac{r^k}{\phi^k(D_yf^{\ss}_{\pmb{\omega}})}:\; k=0,1,\ldots,d\right\}.
\end{split}
\end{equation*}
 Since $y\in S$ is arbitrary and $\Pi^{\ss}(\Sigma)\subset S$, recalling  \begin{equation*}
Z_{{\pmb{\omega}}}^{\ss}(r)=
\inf_{x\in \Sigma} \min\left \{ \frac{r^k}{ \phi^k (D_{\Pi^{\ss} x}f^{\ss}_{{\pmb{\omega}}} ) }:\; k=0, 1,\ldots, d\right\},
\end{equation*}
 it follows that
$$
\mathcal L_{\ell d}\left\{\tt \in B_{\R^{\ell d}}(\ss,\delta)\cap \Delta:\; |\Pi^{\tt}(\bi)-\Pi^{\tt}(\bj)|< r\right\}
\leq \widetilde{C}C(\delta)^de^{ndh(\delta)}Z_{\pmb{\omega}}^{\ss}(r).
$$
This completes the proof of the theorem by letting $c_\delta=\widetilde{C}C(\delta)^d$ and $\psi(\delta)=dh(\delta)$.
\end{proof}

In what follows we prove Lemma \ref{lem-key}.  To this end, we first prove  an elementary geometric lemma.
\begin{lem}
\label{lem-geom}
Let $A$ be a real invertible $d\times d$ matrix. Then for $r_1,r_2>0$,
$$
\mathcal L_{d}\left(\left(A^{-1}B_{\R^d}(0,r_1)\right)\cap B_{\R^d}(0,r_2)\right)\leq 2^d \min\left\{ \frac{r_1^kr_2^{d-k}}{\phi^k(A)}:\; k=0,1,\ldots, d\right\},
$$
where $\phi^s(\cdot)$ is the singular value function defined as in \eqref{e-singular}.
\end{lem}
\begin{proof}
Let $\alpha_1\geq \cdots\geq\alpha_d$ be the singular values of $A$. Clearly the set $$\left(A^{-1}B_{\R^d}(0,r_1)\right)\cap B_{\R^d}(0,r_2)$$ is contained in a rectangular parallelepiped with sides $2\min\{r_1/\alpha_i, r_2\}$, $i=1,\ldots, d$. It follows that
\begin{eqnarray*}
\mathcal L_d\left(\left(A^{-1}B_{\R^d}(0,r_1)\right)\cap B_{\R^d}(0,r_2)\right)&\leq& 2^d\prod_{i=1}^d\min\left\{\frac{r_1}{\alpha_i},\; r_2\right\}\\
&=& 2^d \min\left\{\frac{r_1^kr_2^{d-k}}{\alpha_1\ldots \alpha_k}:\; k=0,1,\ldots, d\right\}\\
&=& 2^d \min\left\{\frac{r_1^kr_2^{d-k}}{\phi^k(A)}:\; k=0,1,\ldots, d\right\}.
\end{eqnarray*}
\end{proof}

\begin{proof}[Proof of Lemma \ref{lem-key}]

Let $\ba=(a_n)_{n=1}^\infty,\bb=(b_n)_{n=1}^\infty\in \Sigma$ with $a_1\neq b_1$. Without loss of generality we assume that
\begin{equation}
\label{e-assump}
 a_1=1 \quad \mbox{ and } \quad b_1=2.
\end{equation}

Define $g:\Delta\to \R^d$ by
$$
g(\tt)=\Pi^{\tt}(\ba)-\Pi^{\tt}(\bb).
$$
 Recall that we have used the notation
 $
\pmb{\mathfrak{t}} =(\mathbf{t}_1, \dots ,\mathbf{t}_\ell )\in \Delta\subset \mathbb{R}^{\ell d }$  with  $$\mathbf{t}_k=(t_{k,1},\ldots, t_{k,d})\in\mathbb{R}^d \quad \mbox{ for all $1\leq k\leq \ell$}.
 $$
 For  $\pmb{\mathfrak{t}} =(\mathbf{t}_1, \dots ,\mathbf{t}_\ell )\in \Delta$ and $k\in \{1,\ldots, \ell\}$, let $\frac{\partial g}{\partial\mathbf{t} _k}(\pmb{\mathfrak{t}} )$ denote  the Jacobian matrix of the following map  from $\R^d$ to $\R^d$:
\[(t_{k,1},\ldots, t_{k, d})\mapsto g(\mathbf{t}_1,\ldots,\mathbf{t}_{k-1}, t_{k,1}, \ldots, t_{k,d}, \mathbf{t}_{k+1}, \ldots,\mathbf{t}_{\ell }).\]

Write $\mathbf{I}=\mathbf{I}_{d}:={\rm diag}(\underbrace{1,\ldots,1}_{d})$. First observe that
for every $n\in \N$ and $\mathbf{i}=(i_k)_{k=1}^\infty\in \Sigma$,
\begin{equation}\label{c78}
   \Pi^{\pmb{\mathfrak{t}} }(\mathbf{i})=
   \mathbf{t}_{i_1}+f_{i_1}
   \left(\mathbf{t}_{i_2}+f_{i_2}
   \left( \mathbf{t}_{i_3}+f_{i_3} \left(
   \dots
   f_{i_{n-1}}\left(\mathbf{t}_{i_n}+f_{i_n}\left(\Pi^{\pmb{\mathfrak{t}} }\sigma^n\mathbf{i}\right)\dots \right)\right)\right)\right).
    \end{equation}
It follows that for $k\in \{1,\ldots, \ell\}$,
\begin{equation}
\label{e-W2}
\begin{split}
 \frac{\partial\Pi^{\pmb{\mathfrak{t}}}
 (\mathbf{a})}{\partial\mathbf{t}_{k}}\!  &=\!
\delta_{k, a_1} \cdot \mathbf{I}+(D_{\Pi^{\pmb{\mathfrak{t}}}(\sigma\mathbf{a})}f_{a_1}^{\pmb{\mathfrak{t}}})
\times
   \\
   & \left[\mathbf{I}\cdot\delta_{k, a_2}+(D_{\Pi^{\pmb{\mathfrak{t}}}
   (\sigma^2\mathbf{a})} f_{a_2}^{\pmb{\mathfrak{t}}})[\mathbf{I}\cdot\delta_{k, a_3}+ (D_{\Pi^{\pmb{\mathfrak{t}}}(\sigma^3\mathbf{a})}f_{a_3}^{\pmb{\mathfrak{t}}})
      [\mathbf{I}\cdot\delta_{k, a_4}+\cdots] ]\right]
   \\
   &=
\delta_{k, a_1} \cdot \mathbf{I}+
\sum\limits_{n \geq 1 \atop a_{n+1}=k}\
\prod_{k=1}^{n}D_{\Pi^{\pmb{\mathfrak{t}} }(\sigma^k\mathbf{a})}f_{a_k},
\end{split}
\end{equation}
where $\delta_{i,j}=1$ if $i=j$ and $0$ otherwise.

By \eqref{e-W2} and  the assumption \eqref{e-assump},
we  see that  for $k\in \{1,\ldots, \ell\}$,
\begin{eqnarray}\label{Wa52}
 \frac{\partial g}{{\partial\mathbf{t}_{k}}}(\pmb{\mathfrak{t}} )\!&=&\!
\frac{\partial \Pi^{\pmb{\mathfrak{t}}}(\mathbf{a})}{\partial\pmb{\mathfrak{t}} _{k}}
-
\frac{\partial \Pi^{\pmb{\mathfrak{t}}}(\mathbf{b})}{\partial\pmb{\mathfrak{t}} _{k}}=
\delta_{k, 1} \cdot \mathbf{I}-\delta_{k, 2} \cdot \mathbf{I}
\!+\!
E_{k}(\pmb{\mathfrak{t}} ) ,
\end{eqnarray}
where
\begin{equation}\label{Wa50}
 E_{k}(\pmb{\mathfrak{t}} ) = \sum\limits_{n \geq 1 \atop a_{n+1}=k}\
\prod_{i=1}^{n}D_{\Pi^{\pmb{\mathfrak{t}} }(\sigma^i\mathbf{a})}f_{a_i}
-
\sum\limits_{n \geq 1 \atop b_{n+1}=k}\
\prod_{i=1}^{n}D_{\Pi^{\pmb{\mathfrak{t}} }(\sigma^i\mathbf{b})}f_{b_i}.
\end{equation}

Recall that   $\rho=\max\limits_{i\neq j}\left(\rho_i+\rho_j\right)<1$ with
  $\rho_i:=\max\limits_{z\in S}\|D_zf_i\|$.

\begin{lem}\label{Wa48} There exists $k^*=k^*(\ba,\bb)\in \{1,2\}$ such that   $\|E_{k^*}(\pmb{\mathfrak{t}} )\| < \rho$ for all $\pmb{\mathfrak{t}} \in \Delta$.
\end{lem}
\begin{proof}
Our argument  is based on an idea of Boris Solomyak which was used to prove a corresponding statement for self-affine IFSs \cite[Theorem 9.1.2]{BSS2021}.

 By \eqref{Wa50}, for each $k\in \{1,2\}$ and  $\pmb{\mathfrak{t}} \in \Delta$,
\begin{equation}\label{Wa47}
  \|E_{k }({\pmb{\mathfrak{t}} })\|
   \leq \sum\limits_{n \geq 1 \atop a_{n+1}=k}\rho_{a_1} \cdots \rho_{a_n}
  +
   \sum\limits_{n \geq 1 \atop b_{n+1}=k}\rho_{b_1} \cdots \rho_{b_n}=:\lambda_k.
\end{equation}
Clearly $\lambda_k$ ($k=1,2$) only depend on $\ba$ and $\bb$.

Notice that
\begin{equation}\label{Wa49}
\begin{split}
\sum\limits_{k =1}^{2} \lambda_k (1-\rho_{k })
 &=
 \sum\limits_{n=1}^{\infty }\rho_{a_1} \cdots \rho_{a_n}(1-\rho_{a_{n+1}})
 +
 \sum\limits_{n=1}^{\infty }\rho_{b_1} \cdots \rho_{b_n}(1-\rho_{b_{n+1}})
 \\
 &=\rho_1+\rho_2\\
 &\leq \rho.\\
  \end{split}
\end{equation}
This implies that one of $\lambda_1,\lambda_2$ is smaller than $\rho$; otherwise, since $\rho_1+\rho_2\leq \rho<1$,  it follows that
$$
\sum\limits_{k =1}^{2} \lambda_k (1-\rho_{k })\geq \rho(1-\rho_1+1-\rho_2)>\rho,
$$
which contradicts \eqref{Wa49}.
 Now set
$$
k^*=\left\{
\begin{array}{ll}
1 & \mbox{ if } \lambda_1<\rho,\\
2 & \mbox{ otherwise}.
\end{array}
\right.
$$
Then $\lambda_{k^*}<\rho$. Since $\lambda_1, \lambda_2$ only depend on $\ba$ and $\bb$, so does  $k^*$. By \eqref{Wa47}, $$\| E_{k^*}(\pmb{\mathfrak{t}} ) \|\leq \lambda_{k^*}<\rho$$ for all $\tt\in \Delta$.
\end{proof}

In what follows, we always  let  $k^*=k^*(\ba, \bb)\in \{1,2\}$ be given as in Lemma \ref{Wa48}.
\begin{lem}\label{Wa45}
For all  $\pmb{\mathfrak{t}} \in \Delta$,
\begin{equation}\label{Wa46}
\left\|\left( \frac{\partial g}{{\partial\mathbf{t}_{k^* }}}(\pmb{\mathfrak{t}} )\right)^{-1}
  \right\| < \frac{1}{1-\rho}.
\end{equation}
\end{lem}
\begin{proof}
  Without loss of generality we assume that $k^*=1$. The proof is similar in the case when $k^*=2$.

  Let $\pmb{\mathfrak{t}} \in \Delta$. By Lemma \ref{Wa48}, $\|E_1(\pmb{\mathfrak{t}} )\|< \rho<1$. Thanks to  \eqref{Wa52},
    \begin{equation}
    \label{e-Wa}
 \frac{\partial g}{{\partial\mathbf{t}_{1 }}}(\pmb{\mathfrak{t}} )
=
 \mathbf{I}-\left(-E_1(\pmb{\mathfrak{t}} )\right),
    \end{equation}
 where   $\mathbf{I}=\mathrm{diag}(\underbrace{1,\ldots,1}_d)$.
Since $\|E_1(\pmb{\mathfrak{t}} )\|< \rho<1$, we see that $\frac{\partial g}{{\partial\mathbf{t}_{1 }}}(\pmb{\mathfrak{t}} )$ is invertible with
 $$
 \left(\frac{\partial g}{{\partial\mathbf{t}_{1 }}}(\pmb{\mathfrak{t}})\right )^{-1}=\mathbf{I}+ \sum_{n=1}^\infty (-E_1(\pmb{\mathfrak{t}} ))^n,
 $$ from which  we obtain that
 \begin{equation}\label{Wa32}
   \left\|\left( \frac{\partial g}{{\partial\mathbf{t}_{1 }}}(\pmb{\mathfrak{t}} )\right)^{-1}
  \right\|\leq 1+\sum_{n=1}^\infty \|E_1(\pmb{\mathfrak{t}} )\|^n\leq  1+\sum_{n=1}^\infty \rho^n=\frac{1}{1-\rho}.
  \end{equation}
  \end{proof}

  Next we
 introduce  two mappings
$T_1, T_2:
  \Delta
  \to\mathbb{R}^{\ell d}$ by
\begin{equation}\label{Wa43}
T_{1}(\tt)=
  \left(
g(\pmb{\mathfrak{t}}),
\mathbf{t}_{2},  \ldots, \mathbf{t}_{\ell }\right), \quad T_{2}(\tt)=
\left(\mathbf{t}_1,
g(\pmb{\mathfrak{t}}),
\mathbf{t}_{3},  \ldots, \mathbf{t}_{\ell }\right),
\end{equation}
where $\tt=(\mathbf{t}_{1},\ldots, \mathbf{t}_{\ell})$.  Recall that $g(\tt)=\Pi^{\tt}(\ba)-\Pi^{\tt}(\bb)$.
\begin{lem}\label{lem: det T'}Let  $k^*=k^*(\ba, \bb)\in \{1,2\}$ be given as in Lemma \ref{Wa48}. Then the following properties hold.
\begin{itemize}
\item[(i)] The mapping $T_{k^*}:\;\Delta\to \R^{\ell d}$ is injective.
\item[(ii)]
For each $\tt\in \Delta$,
\begin{equation} \left| \det\left(\left( D_{\tt}T_{k^*}\right)^{-1}\right)
\right|<\left(\frac{1}{1-\rho}\right)^d.
\end{equation}
\end{itemize}
\end{lem}

\begin{proof}
 Without loss of generality, we may assume that $k^*=1$. Then by  Lemma \ref{Wa48} and \eqref{Wa43},
 \begin{equation}
 \label{e-new1}
 \|E_1(\tt)\|<\rho, \qquad  T_1(\tt)=\left(
g(\pmb{\mathfrak{t}}),
\mathbf{t}_{2},  \ldots, \mathbf{t}_{\ell }\right)
\end{equation} for $\tt=(\mathbf{t}_{1},  \ldots, \mathbf{t}_{\ell })\in \Delta$.  Hence to prove (i), it suffices to show that
for given $\mathbf{t}_{2}, \ldots, \mathbf{t}_{\ell }\in \R^d$ with $\sum_{i=2}^\ell |\mathbf{t}_{2}|^2<r_0^2$, the mapping
$$
\mathbf{t}_1\mapsto g(\mathbf{t}_{1}, \mathbf{t}_{2}, \ldots, \mathbf{t}_{\ell })
$$
is injective on $\Delta_1:=\left\{\mathbf{t}_1\in \R^d:\; |\mathbf{t}_1|<\sqrt{r_0^2-\sum_{i=2}^\ell |\mathbf{t}_{2}|^2}\right\}$.
To this end, define $\psi:\; \Delta_1\to \R^d$ by
$$
\psi(\mathbf{t}_1)=g(\mathbf{t}_{1},  \ldots, \mathbf{t}_{\ell })-\mathbf{t}_1.
$$
Then by \eqref{e-Wa} and \eqref{e-new1}, $$\|D_{\mathbf{t}_1}\psi\|=\left\|\frac{\partial g}{\partial \mathbf{t}_1}(\mathbf{t}_{1},  \ldots, \mathbf{t}_{\ell })-\mathbf{I}\right\|=\|E_1(\mathbf{t}_{1},  \ldots, \mathbf{t}_{\ell })\|<\rho \quad \mbox{for each $\mathbf{t}_{1}\in \Delta_1$}.
$$
 Since $\Delta_1$ is a convex open subset of $\R^d$, by \cite[Theorem 9.19]{Rudin1976} the above inequality implies that
 $$
 |\psi(\mathbf{t}_1)-\psi(\mathbf{s}_1)|\leq \rho |\mathbf{t}_1-\mathbf{s_1}|<|\mathbf{t}_1-\mathbf{s_1}|
 $$
 for all distinct $\mathbf{t}_1, \mathbf{s_1}\in \Delta_1$.  It follows that for distinct $\mathbf{t}_1, \mathbf{s_1}\in \Delta_1$,
 \begin{eqnarray*}
 |g(\mathbf{t}_1)-g(\mathbf{s}_1)|&=&|\psi(\mathbf{t}_1)+\mathbf{t}_1-\psi(\mathbf{s}_1)-\mathbf{s}_1|\\
 &\geq& |\mathbf{t}_1-\mathbf{s}_1|-|\psi(\mathbf{t}_1)-\psi(\mathbf{s}_1)|\\
 &>&0.
\end{eqnarray*}
This proves (i).

To prove (ii), notice that
\[ D_{\tt}T_{1}=\left(\begin{array}{c|cccc}
    \frac{\partial{g}}{\partial{\mathbf{t}_1}}(\tt) & \frac{\partial{g}}{\partial{\mathbf{t}_2}}(\tt)  & \frac{\partial{g}}{\partial{\mathbf{t}_3}}(\tt) & \cdots & \frac{\partial{g}}{\partial{\mathbf{t}_{\ell }}}(\tt) \\\hline
    &   &  &  & \\
  \mathbf{0}_{(\ell -1)d,d}  &   &  & \mathbf{I}_{(\ell -1)d} &  \\
    & & &  &
  \end{array}\right),\]
  where $\mathbf{I}_{(\ell-1)d}:={\rm diag}(\underbrace{1, \dots ,1}_{(\ell-1)d})$ and
  $\mathbf{0}_{(\ell-1)d,d}$ is the $((\ell-1)d)\times d$ all-zero matrix.
  That is,
\[D_{\tt}T_{1}=\left(\begin{array}{@{}c|c@{}}A(\pmb{\mathfrak{t}})& B(\pmb{\mathfrak{t}})\\ \hline
\mathbf{0}_{(\ell-1)d, d}&\mathbf{I}_{(\ell-1)d}\end{array}\right),\]
where $A(\pmb{\mathfrak{t}})$ and $B(\pmb{\mathfrak{t}})$ are given by
\[A(\pmb{\mathfrak{t}})=
\frac{\partial{g}}{\partial{\mathbf{t}_1}}
(\pmb{\mathfrak{t}} ),
\qquad B(\pmb{\mathfrak{t}})=
\left(\frac{\partial{g}}{\partial{\mathbf{t}_2}}(\tt), \dots ,\frac{\partial{g}}{\partial{\mathbf{t}_{\ell }}}(\tt)\right).
 \]
 Hence by Lemma \ref{Wa45}, $A^{-1}(\pmb{\mathfrak{t}} )$ exists and
 \[(D_{\tt}T_{1})^{-1}=\left(\begin{array}{@{}c|c@{}}A^{-1}(\pmb{\mathfrak{t}})&-A^{-1}(\pmb{\mathfrak{t}})\cdot B(\pmb{\mathfrak{t}})\\ \hline
 \mathbf{0}_{(\ell-1)d,d)} &\mathbf{I}_{(\ell-1)d}\end{array}\right).\]
 It follows that
 \begin{equation*}\label{det A^-1,2}
 \det\left((D_{\tt}T_{1})^{-1}\right)=
 \det\left(A^{-1}(\pmb{\mathfrak{t}})\right)
 =\det\left(\left(\frac{\partial g}{\partial\mathbf{t}_1}(\tt)\right)^{-1}\right).
 \end{equation*}
By the Hadamard's inequality (see e.g.~\cite[Corollary 7.8.2]{HornJohnson1985}),
\begin{equation*}\label{Wa41}
\left|\det\left((D_{\tt}T_{1})^{-1}\right)\right|=
\left|\det\left(
  \left(\frac{\partial g}{{\partial\mathbf{t}_{1}}}(\pmb{\mathfrak{t}} )\right)^{-1}\right)\right|
   \leq \left\|\left(\frac{\partial g}{{\partial\mathbf{t}_{1}}}(\pmb{\mathfrak{t}} )\right)^{-1}\right\|^{d}
\leq \left(\frac{1}{1-\rho}\right)^{d},
\end{equation*}
where the last inequality follows from Lemma \ref{Wa45}. This completes the proof of (ii).
\end{proof}
To shorten the notation, from now on we write
\begin{equation}\label{a15}
  C_{*}:=\left(\frac{1}{1-\rho}\right)^{d}.
\end{equation}
Let $\ss=(\mathbf{s}_1,\ldots, \mathbf{s}_\ell)\in \Delta$ and $\delta, r>0$. Let $A$ be a given real invertible $d\times d$ matrix.

Write
\begin{eqnarray*}
E&:=&\left\{\tt\in B_{\R^{\ell d}}(\ss, \delta)\cap \Delta:\; \Pi^{\tt}(\ba)-\Pi^{\tt}(\bb)\in A^{-1}B_{\R^d}(0,r)\right\} \\
&=& \left\{\tt\in B_{\R^{\ell d}}(\ss, \delta)\cap \Delta:\; g(\tt)\in A^{-1}B_{\R^d}(0,r)\right\}.
\end{eqnarray*}
Below we estimate $\mathcal L_{\ell d}(E)$.

 Notice that $\mathcal L_{\ell d}(E)=\mathcal L_{\ell d}\left(T_{k^*}^{-1}(T_{k^*}(E))\right)$. Recall that  by Lemma \ref{lem: det T'}, the mapping $T_{k^*}: \Delta\to \R^{\ell d}$ is injective, and $\det\left((D_{\tt}T_{k^*})^{-1}\right)\leq C_*$ for $\tt\in \Delta$. So by the substitution rule of  multiple integration (see e.g. \cite[Theorem 10.9]{Rudin1976}),
 \begin{equation}
 \label{e-W3}
 \mathcal L_{\ell d}(E)\leq C_*\mathcal L_{\ell d}\left(T_{k^*}(E)\right).
 \end{equation}

 Next we estimate $\mathcal L_{\ell d}\left(T_{k^*}(E)\right)$. Without loss of generality we may assume that $k^*=1$.
 Notice that for each $\tt\in E$,
 $$
 g(\tt)\in A^{-1}B_{\R^d}(0,r);$$
 in the meantime since $\Pi^{\tt}(\ba),\Pi^{\tt}(\bb)\in S$,  it follows that
  $$g(\tt)=\Pi^{\tt}(\ba)-\Pi^{\tt}(\bb)\in
 B_{\R^d}(0, 2{\rm diam}(S)).$$
 Hence,  for each $\tt\in E$,
 $$
 g(\tt)\in \left(A^{-1}B_{\R^d}(0,r)\right)\cap B_{\R^d}(0, 2{\rm diam}(S)).
 $$
 Since $T_1(\tt)=(g(\tt), \mathbf{t_2},\ldots, \mathbf{t}_\ell)$,  it follows that
 $$
 T_{1}(E)\subset F_1\times F_2,
 $$
 where \begin{align*}
 F_1:=&\left(A^{-1}B_{\R^d}(0,r)\right)\cap B_{\R^d}(0, 2{\rm diam}(S)),\\
 F_2:=&\left\{(\mathbf{t}_2,\ldots, \mathbf{t}_\ell)\in \R^{(\ell-1)d}:\; |\mathbf{t}_i-\mathbf{s}_i|<\delta\right\}.
 \end{align*}
Consequently,
 \begin{eqnarray*}
 \mathcal L_{\ell d}\left(T_{1}(E)\right)&\leq & \mathcal L_{d}(F_1)\cdot \mathcal L_{(\ell-1)d}(F_2)\\
 &\leq & 2^d\min\left\{
 \frac{r^k (2{\rm diam}(S))^{d-k}}{\phi^k(A)}:\; k=0,1,\ldots,d
 \right\}\cdot (2\delta)^{(\ell-1)d}\\
 &\leq & u \min\left\{
 \frac{r^k }{\phi^k(A)}:\; k=0,1,\ldots,d
 \right\}
 \end{eqnarray*}
 with $u:=2^{\ell d}\delta^{(\ell-1)d}\max\left\{1, 2^d{\rm diam}(S)^d\right\}$, where we have used Lemma \ref{lem-geom} in the second inequality.  Combining this with \eqref{e-W3} yields that
 $$
 \mathcal L_{\ell d}(E)\leq C_*  \mathcal L_{\ell d}(T_1(E))\leq uC_* \min\left\{
 \frac{r^k }{\phi^k(A)}:\; k=0,1,\ldots,d
 \right\}.
 $$
 This completes the proof of Lemma \ref{lem-key}.
\end{proof}

\section{Translational family of  IFSs generated by a  $C^1$ conformal IFS}
\label{S-conformal}
In this section, we prove the following result.

\begin{thm}
\label{thm-6.1}
Let $\mathcal F=\{f_i:\; S\to S\}_{i=1}^\ell$ be an IFS on a compact set $S\subset \R^d$. Suppose that the following properties hold:
\begin{itemize}
\item[(i)]
The set $S$ is connected,  $S=\overline{{\rm int}(S)}$ and $f_i(S)\subset {\rm int}(S)$ for all $i$.
\item[(ii)] There is a bounded connected open set $U\supset S$ such that each $f_i$ extends to a $C^1$ conformal diffeomorphism $f_i:\; U\to f_i(U)\subset U$ with $$\rho_i:=\sup_{x\in U}\|f_i'(x)\|<1.$$
\item[(iii)] $\max_{i\neq j} \rho_i+\rho_j<1$.
\end{itemize}
Then there is a small $r_0>0$ such that the translational family ${\mathcal F}^{\tt}=\{f_i^{\tt}=f_i+\mathbf{t}_i\}_{i=1}^\ell$, $\tt=(\mathbf{t}_1,\ldots, \mathbf{t}_\ell)\in \Delta:=\{\ss\in \R^{\ell d}:\; |\ss|<r_0\}$, satisfies the GTC with respect to the Lebesgue measure $\mathcal L_{\ell d}$ on $\Delta$.
\end{thm}
\begin{proof}
By the assumptions (i) and (ii), we may pick two open connected  sets $V$ and $W$ (for instance, we may let $V$ and $W$ be the $\epsilon$-neighborhood and $2\epsilon$-neighborhood of $S$, respectively, for a sufficiently small $\delta>0$) such that
$$S\subset V\subset \overline{V}\subset W\subset \overline {W}\subset  U, \mbox{  and}
$$
$$
f_i(\overline{V})\subset V \mbox{ and } f_i(\overline{W})\subset W \quad \mbox{ for all } i.
$$
Then by continuity, we can pick a small $r_0$ such that for all $\tt=(\mathbf{t}_1,\ldots, \mathbf{t}_\ell)\in \R^{\ell d}$ with $|\tt|<r_0$,
$$
f_i^{\tt}(\overline{V})\subset V \mbox{ and } f_i^{\tt}(\overline{W})\subset W \quad \mbox{ for all } i,
$$
where  $f_i^{\tt}:=f_i+\mathbf{t}_i$.  Fix this $r_0$ and set $\Delta=\{\ss\in \R^{\ell d}:\; |\ss|<r_0\}$.
In what follows we prove that the family ${\mathcal F}^{\tt}$, $\tt\in \Delta$, satisfies the GTC with respect to $\mathcal L_{\ell d}$ on $\Delta$.

For $i=1,\ldots, \ell$, define $g_i: \overline{W}\to \R$ by $$g_i(z)=\log \|f_i'(z)\|.$$  Then $g_i$ is continuous on $\overline{W}$ for each $i$. Define $\gamma:\; (0,\infty)\to (0,\infty)$ by
$$
\gamma(u)=\max_{1\leq i\leq \ell} \sup\{|g_i(x)-g_i(y)|:\; x,y\in \overline{W},\; |x-y|\leq u\}.
$$
That is, $\gamma$ is a common continuity modulus of $g_1,\ldots, g_\ell$. Clearly $\lim_{u\to 0}\gamma(u)=0$.
Notice that for $\tt\in \Delta$, $y\in W$ and $\pmb{\omega}\in \Sigma_n$,
$$
\log \|(f^{\tt}_{{\pmb{\omega}}} )'(y)\|=\sum_{k=1}^n g_{\omega_k}(f^{\tt}_{\sigma^k{\pmb{\omega}}}(y)).
$$
 Using similar arguments (with minor changes)  as in Step 1 and Step 2 of the proof of Proposition \ref{pro-3.2}, we can show that the following two properties hold:
\begin{itemize}
\item[(a)] Write for $n\in \N$,
\begin{equation}
\label{cn0*}
C_n:=\sup\left\{ \frac{\|(f^{\tt}_{{\pmb{\omega}}} )'(y)\|}{\|(f^{\tt}_{{\pmb{\omega}}} )'(z)\|}:\; \tt\in \Delta,\; y,z\in W, \;{\pmb{\omega}}\in \Sigma_n\right\}.
\end{equation}
Then
$\lim_{n\to \infty}\frac{1}{n}\log C_n=0$.

\item[(b)] For  $y\in W$, $\ss,\tt\in \Delta$, $n\in \N$ and ${\pmb{\omega}}\in \Sigma_n$,
 \begin{equation}
 \label{e-cn1*}
 \frac{\|(f^{\tt}_{{\pmb{\omega}}} )'(y)\|}{\|(f^{\ss}_{{\pmb{\omega}}} )'(y)\|}\leq \exp\left(n\gamma\left(\frac{|\tt-\ss|}{1-\theta}\right)\right),
 \end{equation}
 where $\theta:=\max_{1\leq i\leq \ell}\rho_i<1$.
 \end{itemize}

Let $\mathcal H$ denote the collection of $C^1$ injective conformal mappings $h: W\to W$ such that $h(\overline{V})\subset V$. The following fact is known (for a proof, see e.g. part 3 of the proof of \cite[Lemma 2.2]{Patzschke1997}):
 there exists a constant $D\in (0,1)$ depending on $V$ and $W$, such that
\begin{equation}
\label{e-6.1}
D\cdot \left(\inf_{z\in W}\| h'(z)\|\right) \cdot |x-y|\leq |h(x)-h(y)|  \quad \mbox{ for all }h\in \mathcal H,\; x,y\in V.
\end{equation}

Now fix $\ss\in \Delta$ and $\delta\in (0,r_0)$. Let $\bi,\bj\in \Sigma$ with $\bi\neq \bj$. Set $$\pmb{\omega}=\bi\wedge \bj\quad \mbox{ and }\quad n=|\pmb{\omega}|.$$
Write $\ba=\sigma^n \bi$ and $\bb=\sigma^n \bj$. Clearly $a_1\neq b_1$. Fix $y\in S$. We claim that
 for $r>0$,
\begin{equation}
\label{e-split*}
\begin{split}
&\left\{\tt \in B_{\R^{\ell d}}(\ss,\delta)\cap \Delta:\; |\Pi^{\tt}(\bi)-\Pi^{\tt}(\bj)|< r\right\}\\
&\mbox{}\; \subset
\left\{\tt\in B_{\R^{\ell d}}(\ss,\delta)\cap \Delta:\; \Pi^{\tt}(\ba)-\Pi^{\tt}(\bb)\in \left(D_yf^{\ss}_{\pmb{\omega}}\right)^{-1}
B_{\R^d}\left(0,c(n,\delta)r\right) \right\},
\end{split}
\end{equation}
where
$$
c(n,\delta):=D^{-1}C_n \exp\left(n\gamma\left(\frac{\delta}{1-\theta}\right)\right){\color{red} >1},
$$
in which $D$ is the constant from \eqref{e-6.1}.

To show \eqref{e-split*}, let $\tt \in B_{\R^{\ell d}}(\ss,\delta)\cap \Delta$ so that $|\Pi^{\tt}(\bi)-\Pi^{\tt}(\bj)|<r$. Notice that
\begin{equation*}
\begin{split}
|\Pi^{\tt}(\bi)-\Pi^{\tt}(\bj)|&=|f^{\tt}_{\pmb{\omega}}(\Pi^{\tt}(\ba))-f^{\tt}_{\pmb{\omega}}(\Pi^{\tt}(\bb))|\\
&\geq D\cdot \left(\inf_{z\in W}\| (f^{\tt}_{\pmb{\omega}})'(z)\|\right) \cdot |\Pi^{\tt}(\ba)-\Pi^{\tt}(\bb)|\quad \mbox{ (by \eqref{e-6.1})}\\
&\geq D(C_n)^{-1} \exp\left(-n\gamma\left(\frac{\delta}{1-\theta}\right)\right)\cdot\| (f^{\ss}_{\pmb{\omega}})'(y)\|\cdot |\Pi^{\tt}(\ba)-\Pi^{\tt}(\bb)|,
\end{split}
\end{equation*}
where in the last inequality we have used \eqref{cn0*} and \eqref{e-cn1*}.  It follows that
$$
|\Pi^{\tt}(\ba)-\Pi^{\tt}(\bb)|\leq \| (f^{\ss}_{\pmb{\omega}})'(y)\|^{-1}\cdot D^{-1}C_n \exp\left(n\gamma\left(\frac{\delta}{1-\theta}\right)\right)\cdot r.
$$
Since $(f^{\ss}_{\pmb{\omega}})'(y)=D_yf^{\ss}_{\pmb{\omega}}$  is a scalar multiple of an orthogonal matrix, the above inequality implies that
$$
\Pi^{\tt}(\ba)-\Pi^{\tt}(\bb)\in  \left(D_yf^{\ss}_{\pmb{\omega}}\right)^{-1}
B_{\R^d}\left(0,c(n,\delta)r\right).
$$
from which  \eqref{e-split*} follows.

 By \eqref{e-split*} and Lemma \ref{lem-key} (which is also valid in this context),
\begin{equation*}
\begin{split}
\mathcal L_{\ell d}&\left\{\tt \in B_{\R^{\ell d}}(\ss,\delta)\cap \Delta:\; |\Pi^{\tt}(\bi)-\Pi^{\tt}(\bj)|< r\right\}\\
&\leq \mathcal L_{\ell d} \left\{\tt\in B_{\R^{\ell d}}(\ss,\delta)\cap \Delta:\; \Pi^{\tt}(\ba)-\Pi^{\tt}(\bb)\in \left(D_yf^{\ss}_{\pmb{\omega}}\right)^{-1}
B_{\R^d}\left(0,c(n,\delta)r\right) \right\}\\
&\leq \widetilde{C}\cdot\min\left\{\frac{c(n,\delta)^k r^k}{\phi^k(D_yf^{\ss}_{\pmb{\omega}})}:\; k=0,1\ldots,d\right\}\\
&\leq \widetilde{C}c(n,\delta)^d\cdot\min\left\{\frac{ r^k}{\phi^k(D_yf^{\ss}_{\pmb{\omega}})}:\; k=0,1\ldots,d\right\}\\
&= \widetilde{C}D^{-d} (C_n)^d \exp\left(nd\gamma\left(\frac{\delta}{1-\theta}\right)\right) \min\left\{\frac{r^k}{\phi^k(D_yf^{\ss}_{\pmb{\omega}})}:\; k=0,1\ldots,d\right\}.
\end{split}
\end{equation*}
Since $y\in S$ is arbitrary and $\Pi^{\ss}(\Sigma)\subset S$, recalling \begin{equation*}
Z_{{\pmb{\omega}}}^{\ss}(r)=
\inf_{x\in \Sigma} \min\left \{ \frac{r^k}{ \phi^k (D_{\Pi^{\ss} x}f^{\ss}_{{\pmb{\omega}}} ) }:\; k=0, 1,\ldots, d\right\},
\end{equation*}
 it follows that
\begin{align*}
\mathcal L_{\ell d}& \left\{\tt \in B_{\R^{\ell d}}(\ss,\delta)\cap \Delta:\; |\Pi^{\tt}(\bi)-\Pi^{\tt}(\bj)|< r\right\}
\\
& \leq \widetilde{C}D^{-d} (C_n)^d \exp\left(nd\gamma\left(\frac{\delta}{1-\theta}\right)\right)Z_{\pmb{\omega}}^{\ss}(r)\\
&\leq c_\delta e^{n\psi(\delta)}Z_{\pmb{\omega}}^{\ss}(r),
\end{align*}
where
$$
c_\delta:= \sup_{n\in \N}\widetilde{C}D^{-d} (C_n)^d e^{-n\delta}<\infty,\quad \psi(\delta):=\delta+d\gamma\left(\frac{\delta}{1-\theta}\right).
$$
Since $\lim_{u\to 0}\gamma(u)=0$, we  see that  $\lim_{\delta\to 0} \psi(\delta)=0$.  Thus $({\mathcal F}^{\tt})_{\tt\in \Delta}$ satisfies the GTC.
\end{proof}

\section{Direct product of parametrized families of $C^1$ IFSs}
\label{S-direct}

In this section we  study the direct product of parametrized families of $C^1$ IFSs (cf. Definition \ref{de-1.5}).
The main result is the following,   stating that the property of GTC  is preserved under the direct product.

\begin{pro}
\label{pro-4.2}
Let $\ell\in \N$ with $\ell\geq 2$. Suppose that for $k=1,\ldots, n$,  $(\mathcal F_k^{t_k})_{t_k\in {\Omega}_k}$ is a parametrized family of
$C^1$ IFSs on $Z_j\subset \R^{q_k}$,  satisfying the GTC with respect to a locally finite Borel measure $\eta_k$ on the metric space $({\Omega}_k, d_{{\Omega}_k})$.  Moreover, suppose all the individual IFSs have $\ell$ contractions. Set
$$
\mathcal F^{(t_1,\ldots, t_n)}=\mathcal F_1^{t_1}\times\cdots\times \mathcal F_n^{t_n},\quad (t_1,\ldots, t_n)\in {\Omega}_1\times\cdots \times {\Omega}_n.
$$
Endow ${\Omega}:={\Omega}_1\times\cdots \times {\Omega}_n$ with the product metric $d_{\Omega}$ as follows:
$$
d_\Omega((t_1,\ldots, t_n), (s_1,\ldots, s_n))=\left(\sum_{k=1}^n d_{{\Omega}_k}(s_k, t_k)^2\right)^{1/2}.
$$
Then the family $\mathcal F^{(t_1,\ldots, t_n)}$, $(t_1,\ldots, t_n)\in {\Omega}_1\times\cdots \times {\Omega}_n$, satisfies the GTC with respect to
$\eta_1\times\cdots\times\eta_n$.
\end{pro}

To prove the above proposition, we need the following.

\begin{lem}
\begin{itemize}
\item[(i)] Let $A$ be a real non-singular $d\times d$ matrix with singular values $\alpha_1\geq \cdots\geq \alpha_d$. Then for each $r>0$,
$$
\min\left\{\frac{r^p}{\phi^p(A)}:\; p=0,1,\ldots,d\right\}=\prod_{i=1}^d\frac{\min\{ \alpha_i, r\}}{\alpha_i},
$$
where $\phi^s(\cdot)$ is the singular value function defined as in  \eqref{e-singular}.
\item[(ii)] For $j=1,\ldots, n$, let $A_j$ be a real non-singular $d_j\times d_j$ matrix. Set
$$
M={\rm diag}(A_1,\ldots, A_n):=\left[
\begin{array}{cccc}
A_1 & {\bf 0} & \cdots & {\bf 0}\\
{\bf 0} & A_2 & \cdots & {\bf 0}\\
\vdots & \vdots &\ddots & {\bf 0}\\
{\bf 0} & {\bf 0} & \cdots & A_n
\end{array}
\right].
$$
Then
\begin{equation}
\label{e-4.1}
\begin{split}
\min&\left\{\frac{r^p}{\phi^p(M)}:\; p=0,1,\ldots,d_1+\cdots+d_n\right\}\\
&=\prod_{i=1}^n \min\left\{\frac{r^p}{\phi^p(A_i)}:\; p=0,1,\ldots,d_i\right\}.
\end{split}
\end{equation}
\end{itemize}
\end{lem}
\begin{proof}
The proof of (i) is direct and simple. We leave it to the reader as an exercise.  Part (ii) is just a consequence of (i), using the fact  that the set of
 singular values (including the multiplicity) of $M$ are precisely the union of  those of $A_i$, $i=1,\ldots, n$.
\end{proof}

Now we are ready to prove Proposition \ref{pro-4.2}.
 \begin{proof}[Proof of Proposition \ref{pro-4.2}]
 Write
 $$\mathcal F_k^{t_k}=\{f_{i,k}^{t_k}\}_{i=1}^\ell,\qquad t_k\in {\Omega}_k,\; k=1,\ldots,n.$$
 For $\pmb{\omega}=\omega_1\ldots \omega_m\in \Sigma_*$, write $f_{{\pmb{\omega}},k}^{t_k}=f_{\omega_1,k}^{t_k}\circ\cdots\circ f_{\omega_m,k}^{t_k}$. Let $\Pi^{t_k}_k$ denote the coding map associated with the IFS $\mathcal F_k^{t_k}$, and $\Pi^{(t_1,\ldots, t_n)}$  the coding map associated with the IFS $\mathcal F^{(t_1,\ldots, t_n)}$.  According to the GTC assumption on the families $(\mathcal F_k^{t_k})_{t_k\in {\Omega}_k}$, $k=1,\ldots, n$, there exist $\delta_0>0$ and a function $\psi:\; (0, \delta_0)\to [0,\infty)$ with $\lim_{\delta\to 0}\psi(\delta)=0$ such that for every $\delta\in (0, \delta_0)$ and $(s_1,\ldots, s_n)\in {\Omega}_1\times\cdots\times {\Omega}_n$,  there is  $C=C(\delta,s_1,\ldots, s_n)>0$ satisfying the following: for each $k\in \{1,\ldots, n\}$,   distinct $\bi,\bj\in \Sigma$ and $r>0$,
 \begin{equation}
\label{e-f2*}
\begin{split}
\eta_k&\left\{t_k\in B_{{\Omega}_k}(s_k, \delta):  \; |\Pi^{t_k}_k({\bf i})-\Pi^{t_k}_k({\bf j})  |< r\right\}\\
&\leq C e^{|{\bf i}\wedge {\bf j}| \psi(\delta)} \inf_{x\in \Sigma} \min\left\{
\frac{r^p}{\phi^p\left(D_{\Pi^{t_k}_kx}f_{\bi\wedge \bj, k}^{t_k}\right)}:\; p=0,1,\ldots, q_k
\right\},\end{split}
\end{equation}
where $B_{{\Omega}_k}(s_k, \delta)$ stands for the closed ball in ${\Omega}_k$ of radius $\delta$ centered at $s_k$.  Writing
$\mathbf{t}=(t_1,\ldots, t_n)$,  $\mathbf{s}=(s_1,\ldots, s_n)$  and using \eqref{e-f2*},
\begin{equation*}
\begin{split}
\eta_1&\times \cdots\times \eta_n\left\{ \mathbf{t}\in  B_{{\Omega}}(\mathbf{s}, \delta):  \; |\Pi^{\mathbf{t}}({\bf i})-\Pi^{\mathbf{t}}({\bf j})  |< r\right\}\\
&\leq \prod_{k=1}^n\eta_k\left\{t_k\in B_{{\Omega}_k}(s_k, \delta):  \; |\Pi^{t_k}_k({\bf i})-\Pi^{t_k}_k({\bf j})  |< r\right\}\\
&\leq \prod_{k=1}^n \left(C e^{|{\bf i}\wedge {\bf j}| \psi(\delta)} \inf_{x\in \Sigma} \min\left\{
\frac{r^p}{\phi^p\left(D_{\Pi^{s_k}_kx}f_{\bi\wedge \bj, k}^{s_k}\right)}:\; p=0,1,\ldots, q_k
\right\}\right)\\
&\leq C^n e^{n|{\bf i}\wedge {\bf j}| \psi(\delta)}\inf_{x\in \Sigma}\prod_{k=1}^n \min\left\{
\frac{r^p}{\phi^p\left(D_{\Pi^{s_k}_kx}f_{\bi\wedge \bj, k}^{s_k}\right)}:\; p=0,1,\ldots, q_k
\right\}\\
&=C^n e^{n|{\bf i}\wedge {\bf j}| \psi(\delta)}\inf_{x\in \Sigma}\min \left\{
\frac{r^p}{\phi^p\left(D_{\Pi^{\mathbf{t}}x}f_{\bi\wedge \bj}^{\mathbf{s}}\right)}:\; p=0,1,\ldots, q_1+\cdots+q_n
\right\}\\
&=C^n e^{n|{\bf i}\wedge {\bf j}| \psi(\delta)}Z_{\bi\wedge\bj}^{\mathbf{s}}(r),
\end{split}
\end{equation*}
where we have used \eqref{e-4.1} in the second last  equality. Hence the family $\mathcal F^{\mathbf{t}}$, $\mathbf{t}\in {\Omega}$, satisfies the GTC with respect to the measure $\eta_1\times\cdots\times \eta_n$, where the involved constant and the function in the definition of  GTC are $C^n$ and $n\psi(\cdot)$, respectively.
\end{proof}

\section{The proof of Theorem \ref{thm-1.2} and final questions}
\label{S-thm-1.2}

Now we are ready to prove Theorem \ref{thm-1.2}.
\begin{proof}[Proof of Theorem \ref{thm-1.2}]
{\color{red} This} follows directly  by combining Theorems   \ref{thm-3.3}, \ref{thm-6.1} and Proposition \ref{pro-4.2}.
\end{proof}

Below we  list a few `folklore' open questions on the dimension of  the  attractors of $C^1$ IFSs. One may formulate the corresponding questions on the dimension of push-forwards of ergodic invariant measures on the attractors.

\medskip
\noindent {\bf  Question 1}.  Is it true that for  every $C^1$ IFS $\mathcal F=\{f_i\}_{i=1}^\ell$ on $\R^d$ satisfying \eqref{e-ratio}, there is a neighborhood $\Delta$ of ${\bf 0}$ in $\R^{\ell d}$ such that for $\mathcal L_{\ell d}$-a.e.~$\tt=({\bf t}_1,\ldots, {\bf t}_\ell)\in \Delta$,
$$
\dim_HK^{\tt}=\dim_BK^{\tt}=\min\{\dim_S \mathcal F^{\tt}, \; d\}?
$$
where $K^{\tt}$ is the attractor of the IFS ${\mathcal F}^{\tt}=\{f_i+{\bf t}_{i}\}_{i=1}^\ell$.

\medskip
\noindent {\bf  Question 2}. Do we have
$$
\dim_HK=\dim_BK=\min\{\dim_S \mathcal F, \; d\}
$$
for the attractor $K$ of a ``generic'' $C^1$ IFS $\mathcal F$ on $\R^d$ (in an appropriate  sense)?
\bigskip

\noindent {\bf Acknowledgements}.  The authors would like to thank the referee for helpful comments and suggestions. They also thank Zhou Feng for catching some typos. The research of Feng was partially supported by a HKRGC
GRF grant and the Direct Grant for Research in CUHK. The research of Simon was
partially supported by the grant OTKA K104745. Significant part of the research
was done during Simon's visit to CUHK which was supported by a HKRGC GRF
grant.

\end{document}